\documentclass[12pt]{amsart}

\usepackage[matrix,arrow,curve,frame]{xy}
\usepackage{amsmath,amsthm,amssymb,enumerate}
\usepackage{latexsym}
\usepackage{amscd}
\usepackage[colorlinks=false]{hyperref}
\usepackage{euscript}

\newtheorem{thm}[subsection]{Theorem}
\newtheorem{defn}[subsection]{Definition}
\newtheorem{prop}[subsection]{Proposition}

\newtheorem{cor}[subsection]{Corollary}
\newtheorem{lemma}[subsection]{Lemma}

\newtheorem{remark}[subsection]{Remark}

\theoremstyle{definition}

\numberwithin{equation}{section}

\def\C{{\bf C}}

\def\cP{{\cal P}}
\def\cA{{\cal A}}

\def\ra{\rightarrow}
\def\bra{\langle}
\def\ket{\rangle}

\def\C{{\bf C}}

\def\cA{{\mathcal A}}
\def\cB{{\mathcal B}}

\def\cD{{\mathcal D}}
\def\cE{{\mathcal E}}
\def\cF{{\mathcal F}}

\def\cI{{\mathcal I}}
\def\cJ{{\mathcal J}}

\def\cM{{\mathcal M}}

\def\cP{{\mathcal P}}

\def\cR{{\mathcal R}}
\def\cS{{\mathcal S}}

\def\cV{{\mathcal V}}
\def\cW{{\mathcal W}}

\def\gg{{\mathfrak g}}

\def\gl{{\mathfrak l}}

\def\gs{{\mathfrak s}}

\newcommand{\affine}[1]{\widehat{#1}}
\newcommand{\alg}[1]{\mathfrak{#1}}
\newcommand{\SLA}[2]{\alg{#1} \left( #2 \right)}

\newcommand{\AKMSA}[3]{\affine{\alg{#1}} \left( #2 \middle\vert #3 \right)}
\newcommand{\SLSA}[3]{\alg{#1} \left( #2 \middle\vert #3 \right)}

\newfont{\german}{eufm10}

\begin{document}
\pagestyle{plain}

\title
{The super $\cW_{1+\infty}$ algebra with integral central charge}

\author{Thomas Creutzig and Andrew R. Linshaw}

\address{Department of Mathematics, University of Alberta}
\email{creutzig@ualberta.ca}

\address{Department of Mathematics, University of Denver}
\email{andrew.linshaw@du.edu}

{\abstract \noindent The Lie superalgebra $\cS\cD$ of regular differential operators on the super circle has a universal central extension $\widehat{\cS\cD}$. For each $c\in \mathbb{C}$, the vacuum module $\cM_c(\widehat{\cS\cD})$ of central charge $c$ admits a vertex superalgebra structure, and $\cM_c(\widehat{\cS\cD}) \cong \cM_{-c}(\widehat{\cS\cD})$. The irreducible quotient $\cV_c(\widehat{\cS\cD})$ of the vacuum module is known as the super $\cW_{1+\infty}$ algebra. We show that for each integer $n>0$, $\cV_{n}(\widehat{\cS\cD})$ has a minimal strong generating set consisting of $4n$ fields, and we identify it with a $\cW$-algebra associated to the purely odd simple root system of $\gg\gl(n|n)$. Finally, we realize $\cV_{n}(\widehat{\cS\cD})$ as the limit of a family of commutant vertex algebras that generically have the same graded character and possess a minimal strong generating set of the same cardinality.}

\keywords{invariant theory; vertex algebra; deformation; orbifold construction; strong finite generation; $\cW$-algebra}
\maketitle
\section{Introduction}

Let $\cD$ denote the Lie algebra of regular differential operators on the circle. It has a universal central extension 
$\hat{\cD} = \cD \oplus \mathbb{C}\kappa$ which was introduced by Kac and Peterson in \cite{KP}. Although $\hat{\cD}$ admits a principal 
$\mathbb{Z}$-gradation and triangular decomposition, its representation theory is nontrivial because the graded pieces are all 
infinite-dimensional. The important problem of constructing and classifying the {\it quasifinite} irreducible, highest-weight 
representations (i.e., those with finite-dimensional graded pieces) was solved by Kac and Radul in \cite{KRI}. Explicit constructions of 
these modules were given in terms of the representation theory of $\widehat{\gg\gl}(\infty,R_m)$, which is a central extension of the Lie algebra of infinite matrices over $R_m = \mathbb{C}[t]/(t^{m+1})$ having only finitely many nonzero diagonal elements. 
The authors also classified all such $\hat{\cD}$-modules which are unitary.

In \cite{FKRW}, the representation theory of $\hat{\mathcal{D}}$ was developed from the point of view of vertex algebras. For each $c\in\mathbb{C}$, $\hat{\mathcal{D}}$ admits a module $\mathcal{M}_c$ called the {\it vacuum module}, which is a vertex algebra freely generated by fields $J^l$ of weight $l+1$, for $l\geq 0$. The highest-weight representations of $\hat{\mathcal{D}}$ are in one-to-one correspondence with the highest-weight representations of $\mathcal{M}_c$. The irreducible quotient of $\mathcal{M}_c$ by its maximal graded, proper $\hat{\cD}$-submodule $\cI_c$ is a simple vertex algebra, and is often denoted by $\mathcal{W}_{1+\infty,c}$. These algebras have been studied extensively in both the physics and mathematics literature (see for example \cite{AFMO}\cite{ASV}\cite{BS}\cite{CTZ}\cite{FKRW}\cite{KRII}), and they play an important role the theory of integrable systems. The above central extension is normalized so that $\cM_c$ is reducible if and only if $c\in \mathbb{Z}$. It was shown in \cite{FKRW} that for every integer $n\geq 1$, $\cI_n$ is generated as a vertex algebra ideal by a singular vector of weight $n+1$, and 
\begin{equation}\label{FKRW} \cW_{1+\infty,n} \cong \cW(\gg\gl(n)).\end{equation} 
In particular, $\cW_{1+\infty,n}$ has a minimal strong generating set consisting of a field in each weight $1,2,\dots, n$. The case of negative integral central charge is more complicated. It was shown in \cite{LI} that $\cI_{-n}$ is generated by a singular vector of weight $(n+1)^2$, and $\cW_{1+\infty,-n}$ has a minimal strong generating set consisting of a field in each weight $1,2,\dots, n^2+2n$. Wang showed in \cite{W} that $\cW_{1+\infty,-1}$ is isomorphic to $\cW(\gg\gl(3))$, but for $n>1$ it is not known if $\cW_{1+\infty,-n}$ can be identified with a standard $\cW$-algebra.

The super analogue of $\cD$ is the Lie superalgebra $\cS\cD$ of regular differential operators on the super circle $S^{1|1}$. 
As above, it has a universal central extension $\widehat{\cS\cD}$, and for each $c\in \mathbb{C}$, $\widehat{\cS\cD}$ admits a vacuum module $\cM_c(\widehat{\cS\cD})$. 
This module has a vertex superalgebra structure, and is freely generated by fields 
$$\{J^{0,k}, J^{1,k}, J^{+,k}, J^{-,k}|\ k\geq 0\}$$ 
of weights $k+1, k+1, k+1/2, k+3/2$, respectively. Unlike the modules $\cM_c$ which are all distinct, $\cM_c(\widehat{\cS\cD}) \cong \cM_{-c}(\widehat{\cS\cD})$ for all $c$. There are actions of the affine vertex superalgebra $V_c(\gg\gl(1|1))$ and the $N=2$ superconformal algebra $\cA_c$ of central charge $c$ on $\cM_c(\widehat{\cS\cD})$. Moreover, $\{J^{0,k}|\ k\geq 0\}$ and $\{J^{1,k}|\ k\geq 0\}$ generate copies of $\cM_c$ and $\cM_{-c}$, respectively, which form a Howe pair (i.e., a pair of mutual commutants) inside $\cM_c(\widehat{\cS\cD})$. 

The super $\cW_{1+\infty}$ algebra $\cV_c(\widehat{\cS\cD})$ is the unique irreducible quotient of $\cM_c(\widehat{\cS\cD})$ by its maximal proper graded $\widehat{\cS\cD}$-submodule $\cS\cI_c$. We denote the map $\cM_c(\widehat{\cS\cD}) \ra \cV_c(\widehat{\cS\cD})$ by 
$\pi_c$, and we denote $\pi_c(J^{a,k})$ by $j^{a,k}$ for $a = 0,1,\pm$. There are induced actions of $V_c(\gg\gl(1|1))$ and $\cA_c$ on $\cV_c(\widehat{\cS\cD})$, as well as mutually commuting copies of $\cW_{1+\infty,c}$ and $\cW_{1+\infty,-c}$ inside $\cV_c(\widehat{\cS\cD})$. 

For $n\in \mathbb{Z}$, $\cM_n(\widehat{\cS\cD})$ is reducible and the structure of $\cV_n(\widehat{\cS\cD})$ is nontrivial. Our main goal in this paper is to elucidate this structure. Since $\cV_n(\widehat{\cS\cD}) \cong \cV_{-n}(\widehat{\cS\cD})$ and $V_0(\widehat{\cS\cD}) \cong \mathbb{C}$, it suffices to consider the case $n\geq 1$. This problem was posed by Cheng and Wang; see Problem 3 at the end of \cite{CW}. Our starting point is a free field realization of $\cV_n(\widehat{\cS\cD})$ due to Awata, Fukuma, Matsuo, and Odake as the $GL_n$-invariant subalgebra of the $bc\beta\gamma$-system $\cF$ of rank $n$ \cite{AFMO}. We will show that $\cS\cI_n$ is generated as a vertex algebra ideal by a singular vector of weight $n+1/2$, and that $\cV_n(\widehat{\cS\cD})$ has a minimal strong generating set consisting of the following $4n$ fields: 
$$\{j^{0,k}, j^{1,k}, j^{+,k}, j^{-,k}|\ k=0,\dots, n-1\}.$$ Essentially, these results are formal consequences of Weyl's first and second fundamental theorems of invariant theory for the standard representation of $GL_n$.

Next we show that $\cV_n(\widehat{\cS\cD})$ admits a deformation as the limit of a family of commutant vertex algebras. The rank $n$ $bc\beta\gamma$-system $\cF$ has a natural action of $V_0(\gg\gl(n))$, and we obtain a diagonal homomorphism $V_k(\gg\gl(n)) \ra V_k(\gg\gl(n))\otimes \cF$ for all $k$. We define
$$ \cB_{n,k}= \text{Com}(V_k(\gg\gl(n)), V_k(\gg\gl(n)) \otimes \cF).$$ 
We have $\lim_{k\ra \infty} \cB_{n,k} = \cV_n(\widehat{\cS\cD})$, and for generic values of $k$, $\cB_{n,k}$ has a minimal strong generating set consisting of $4n$ generators and has the same graded character as $\cV_n(\widehat{\cS\cD})$.

Next, we consider a $\cW$-algebra $\cW_{n,k}$ associated naturally to the Lie superalgebra 
$\gg\gl(n|n)$, for $k\in \mathbb{C}$. It is defined as a certain subalgebra of the joint kernel of screening operators corresponding to the purely odd simple root system of $\gg\gl(n|n)$. We expect that $\cW_{n,k}$ coincides with the joint kernel of the screening operators, although we are unable to prove this at present. The $\cW$-algebras of simple affine Lie (super)algebras $\hat\gg$ \cite{FF1}\cite{FF2}\cite{KRW} are defined via the quantum Hamiltonian reduction, which is a certain semi-infinite cohomology. These $\cW$-algebras are associated to the principal embedding of $\gs\gl(2)$ in $\gg$, and they usually can also be realized as the joint kernel of screening operators corresponding to a simple root system of $\gg$. A simple root system of a Lie superalgebra is not unique, and in our case it turns out that a purely odd simple root system is most suitable. We will show that 
\begin{equation} \label{fkrwana} \cV_n(\widehat{\cS\cD})\cong \lim_{k\ra \infty} \cW_{n,k},\end{equation} 
and we regard this as an analogue of the isomorphism $\cW_{1+\infty,n} \cong \cW(\gg\gl(n))$.

Recall that for $n\geq 1$, $\cW_{1+\infty,n}$ and $\cW_{1+\infty,-n}$ have minimal strong generating sets $\{j^{0,k}|\ 0\leq k < n\}$ and $\{j^{1,k}|\ 0\leq k < n^2+2n\}$, respectively, and that only $\cW_{1+\infty,n}$ is known to be a standard $\cW$-algebra. We have identified $\cW_{1+\infty,n}$ and $\cW_{1+\infty,-n}$ as a Howe pair inside $\cV_n(\widehat{\cS\cD})$, which as we have seen is a $\cW$-algebra associated to $\gg\gl(n|n)$. As subalgebras of $\cV_n(\widehat{\cS\cD})$ they are on a similar footing, and in some sense it is more natural to consider them as a Howe pair inside $\cV_n(\widehat{\cS\cD})$, rather than as independent objects. 

Finally, in the case $n=2$, we find a minimal strong generating set for $\cW_{2,k}$ consisting of eight fields, and we show by explicit computation that $\cW_{2,k+2}$ has the same operator product algebra as $\cB_{2,k}$. More generally, we conjecture that $\cW_{n,k+n}$ is isomorphic to $\cB_{n,k}$ for all $k$ and $n$.

There is also a commutant realization of the deformable family of $\cW$-algebras of $\gs\gl(n)$, namely $\text{Com}(V_{k+1}(\gs\gl(n)),V_k(\gs\gl(n))\otimes V_1(\gs\gl(n)))$ \cite{BS}. In physics, for positive integer $k$, the corresponding conformal field theories are called $\cW_n$ minimal models, and they have received much attention recently as tentative dual theories to three dimensional higher spin gravity \cite{GG}. The supersymmetric analogue \cite{CHR} has the $\cW$-superalgebra of $\gs\gl(n+1|n)$ as coset algebra whose twisted algebra in turn is argued to be related to the $\cW$-superalgebra of $\gg\gl(n|n)$ \cite{I}.

\section{Vertex algebras}
In this section, we define vertex algebras, which have been discussed from various different points of view in the 
literature \cite{B}\cite{FBZ}\cite{FHL}\cite{FLM}\cite{K}\cite{LiI}\cite{LZ}. 
We will follow the formalism developed in \cite{LZ} and partly in \cite{LiI}. Let $V=V_0\oplus V_1$ be a super vector space 
over $\mathbb{C}$, and let $z,w$ be formal variables. By $\text{QO}(V)$, we mean the space of all linear maps 
$$V\ra V((z)):=\{\sum_{n\in\mathbb{Z}} v(n) z^{-n-1}|
v(n)\in V,\ v(n)=0\ \text{for} \ n>\!\!>0 \}.$$ Each element $a\in \text{QO}(V)$ can be
uniquely represented as a power series
$$a=a(z):=\sum_{n\in\mathbb{Z}}a(n)z^{-n-1}\in \text{End}(V)[[z,z^{-1}]].$$ We
refer to $a(n)$ as the $n^{\text{th}}$ {\it Fourier mode} of $a(z)$. Each $a\in
\text{QO}(V)$ is of the shape $a=a_0+a_1$ where $a_i:V_j\ra V_{i+j}((z))$ for $i,j\in\mathbb{Z}/2\mathbb{Z}$, and we write 
$|a_i| = i$.

On $\text{QO}(V)$ there is a set of nonassociative bilinear operations
$\circ_n$, indexed by $n\in\mathbb{Z}$, which we call the $n^{\text{th}}$ circle
products. For homogeneous $a,b\in \text{QO}(V)$, they are defined by
$$
a(w)\circ_n b(w)=\text{Res}_z a(z)b(w)~\iota_{|z|>|w|}(z-w)^n-
(-1)^{|a||b|}\text{Res}_z b(w)a(z)~\iota_{|w|>|z|}(z-w)^n.
$$
Here $\iota_{|z|>|w|}f(z,w)\in\mathbb{C}[[z,z^{-1},w,w^{-1}]]$ denotes the
power series expansion of a rational function $f$ in the region
$|z|>|w|$. We usually omit the symbol $\iota_{|z|>|w|}$ and just
write $(z-w)^{-1}$ to mean the expansion in the region $|z|>|w|$,
and write $-(w-z)^{-1}$ to mean the expansion in $|w|>|z|$. It is
easy to check that $a(w)\circ_n b(w)$ above is a well-defined
element of $\text{QO}(V)$.

The nonnegative circle products are connected through the {\it
operator product expansion} (OPE) formula.
For $a,b\in \text{QO}(V)$, we have \begin{equation}\label{opeform} a(z)b(w)=\sum_{n\geq 0}a(w)\circ_n
b(w)~(z-w)^{-n-1}+:a(z)b(w):\ ,\end{equation} which is often written as
$a(z)b(w)\sim\sum_{n\geq 0}a(w)\circ_n b(w)~(z-w)^{-n-1}$, where
$\sim$ means equal modulo the term $$
:a(z)b(w):\ =a(z)_-b(w)\ +\ (-1)^{|a||b|} b(w)a(z)_+.$$ Here
$a(z)_-=\sum_{n<0}a(n)z^{-n-1}$ and $a(z)_+=\sum_{n\geq
0}a(n)z^{-n-1}$. Note that $:a(w)b(w):$ is a well-defined element of
$\text{QO}(V)$. It is called the {\it Wick product} of $a$ and $b$, and it
coincides with $a\circ_{-1}b$. The other negative circle products
are related to this by
$$ n!~a(z)\circ_{-n-1}b(z)=\ :(\partial^n a(z))b(z):\ ,$$
where $\partial$ denotes the formal differentiation operator
$\frac{d}{dz}$. For $a_1(z),\dots ,a_k(z)\in \text{QO}(V)$, the $k$-fold
iterated Wick product is defined to be
\begin{equation}\label{iteratedwick} :a_1(z)a_2(z)\cdots a_k(z):\ =\ :a_1(z)b(z):,\end{equation}
where $b(z)=\ :a_2(z)\cdots a_k(z):$. We often omit the formal variable $z$ when no confusion can arise.

The set $\text{QO}(V)$ is a nonassociative algebra with the operations
$\circ_n$, which satisfy $1\circ_n a=\delta_{n,-1}a$ for
all $n$, and $a\circ_n 1=\delta_{n,-1}a$ for $n\geq -1$. In particular, $1$ behaves as a unit with respect to $\circ_{-1}$. 
A linear subspace $\cA\subset \text{QO}(V)$ containing $1$ which is closed under the circle products will be called a 
{\it quantum operator algebra} (QOA). Note that $\cA$ is closed under $\partial$
since $\partial a=a\circ_{-2}1$. Many formal algebraic
notions are immediately clear: a homomorphism is just a linear
map that sends $1$ to $1$ and preserves all circle products; a module over $\cA$ is a
vector space $M$ equipped with a homomorphism $\cA\rightarrow
\text{QO}(M)$, etc. A subset $S=\{a_i|\ i\in I\}$ of $\cA$ is said to generate $\cA$ if every element $a\in\cA$ can be written as a 
linear combination of nonassociative words in the letters $a_i$, $\circ_n$, for
$i\in I$ and $n\in\mathbb{Z}$. We say that $S$ {\it strongly generates} $\cA$ if every $a\in\cA$ can be written as a linear combination 
of words in the letters $a_i$, $\circ_n$ for $n<0$. Equivalently, $\cA$ is spanned by the collection 
$\{ :\partial^{k_1} a_{i_1}(z)\cdots \partial^{k_m} a_{i_m}(z):| ~i_1,\dots,i_m \in I,~ k_1,\dots,k_m \geq 0\}$.

We say that $a,b\in \text{QO}(V)$ {\it quantum commute} if $(z-w)^N
[a(z),b(w)]=0$ for some $N\geq 0$. Here $[,]$ denotes the super bracket. This condition implies that $a\circ_n b = 0$ for $n\geq N$, 
so (\ref{opeform}) becomes a finite sum. A {\it commutative quantum operator algebra} (CQOA) is a QOA whose elements pairwise quantum 
commute. Finally, the notion of a CQOA is equivalent to the notion of a vertex algebra. Every CQOA $\cA$ is itself a faithful 
$\cA$-module, called the {\it left regular
module}. Define
$$\rho:\cA\rightarrow \text{QO}(\cA),\qquad a\mapsto\hat a,\qquad \hat
a(\zeta)b=\sum_{n\in\mathbb{Z}} (a\circ_n b)~\zeta^{-n-1}.$$ Then $\rho$ is an injective QOA homomorphism,
and the quadruple of structures $(\cA,\rho,1,\partial)$ is a vertex
algebra in the sense of \cite{FLM}. Conversely, if $(V,Y,{\bf 1},D)$ is
a vertex algebra, the collection $Y(V)\subset \text{QO}(V)$ is a
CQOA. {\it We will refer to a CQOA simply as a
vertex algebra throughout the rest of this paper}.

Let $\cR$ be the category of vertex algebras $\cA$ equipped with a $\mathbb{Z}_{\geq 0}$-filtration
\begin{equation} \cA_{(0)}\subset\cA_{(1)}\subset\cA_{(2)}\subset \cdots,\qquad \cA = \bigcup_{k\geq 0}
\cA_{(k)}\end{equation} such that $\cA_{(0)} = \mathbb{C}$, and for all
$a\in \cA_{(k)}$, $b\in\cA_{(l)}$, we have
\begin{equation} \label{goodi} a\circ_n b\in\cA_{(k+l)},\qquad \text{for}\
n<0,\end{equation}
\begin{equation} \label{goodii} a\circ_n b\in\cA_{(k+l-1)}, \qquad \text{for}\
n\geq 0.\end{equation}
Elements $a(z)\in\cA_{(d)}\setminus \cA_{(d-1)}$ are said to have degree $d$.

Filtrations on vertex algebras satisfying (\ref{goodi})-(\ref{goodii})~were introduced in \cite{LiII}, and are known as {\it good increasing filtrations}. Setting $\cA_{(-1)} = \{0\}$, the associated graded object $\text{gr}(\cA) = \bigoplus_{k\geq 0}\cA_{(k)}/\cA_{(k-1)}$ is a $\mathbb{Z}_{\geq 0}$-graded associative, (super)commutative algebra with a
unit $1$ under a product induced by the Wick product on $\cA$. For each $r\geq 1$ we have the projection \begin{equation} \phi_r: \cA_{(r)} \ra \cA_{(r)}/\cA_{(r-1)}\subset \text{gr}(\cA).\end{equation} 
Moreover, $\text{gr}(\cA)$ has a derivation $\partial$ of degree zero
(induced by the operator $\partial = \frac{d}{dz}$ on $\cA$), and
for each $a\in\cA_{(d)}$ and $n\geq 0$, the operator $a\circ_n$ on $\cA$
induces a derivation of degree $d-k$ on $\text{gr}(\cA)$, which we denote by $a(n)$. Here $$k  = \text{sup} \{ j\geq 1|~ \cA_{(r)}\circ_n \cA_{(s)}\subset \cA_{(r+s-j)},~\forall r,s,n\geq 0\},$$ as in \cite{LL}. These derivations give $\text{gr}(\cA)$ the structure of a vertex Poisson algebra \cite{FBZ}.

The assignment $\cA\mapsto \text{gr}(\cA)$ is a functor from $\cR$ to the category of $\mathbb{Z}_{\geq 0}$-graded (super)commutative rings with a differential $\partial$ of degree zero, which we will call $\partial$-rings. A $\partial$-ring is just an {\it abelian} vertex algebra, that is, a vertex algebra $\cV$ in which $[a(z),b(w)] = 0$ for all $a,b\in\cV$. A $\partial$-ring $A$ is said to be generated by a subset $\{a_i|~i\in I\}$ if $\{\partial^k a_i|~i\in I, k\geq 0\}$ generates $A$ as a graded ring. The key feature of $\cR$ is the following reconstruction property \cite{LL}:

\begin{lemma}\label{reconlem}Let $\cA$ be a vertex algebra in $\cR$ and let $\{a_i|~i\in I\}$ be a set of generators for $\text{gr}(\cA)$ as a $\partial$-ring, where $a_i$ is homogeneous of degree $d_i$. If $a_i(z)\in\cA_{(d_i)}$ are vertex operators such that $\phi_{d_i}(a_i(z)) = a_i$, then $\cA$ is strongly generated as a vertex algebra by $\{a_i(z)|~i\in I\}$.\end{lemma}

As shown in \cite{LI}, there is a similar reconstruction property for kernels of surjective morphisms in $\cR$. Let $f:\cA\rightarrow \cB$ be a morphism in $\cR$ with kernel $\cJ$, such that $f$ maps $\cA_{(k)}$ onto $\cB_{(k)}$ for all $k\geq 0$. The kernel $J$ of the induced map $\text{gr}(f): \text{gr}(\cA)\rightarrow \text{gr}(\cB)$ is a homogeneous $\partial$-ideal (i.e., $\partial J \subset J$). A set $\{a_i|~i\in I\}$ such that $a_i$ is homogeneous of degree $d_i$ is said to generate $J$ as a $\partial$-ideal if $\{\partial^k a_i|~i\in I,~k\geq 0\}$ generates $J$ as an ideal.

\begin{lemma} \label{idealrecon} Let $\{a_i| i\in I\}$ be a generating set for $J$ as a $\partial$-ideal, where $a_i$ is homogeneous of degree $d_i$. Then there exist vertex operators $a_i(z)\in \cA_{(d_i)}$ with $\phi_{d_i}(a_i(z)) = a_i$, such that $\{a_i(z)|\ i\in I\}$ generates $\cJ$ as a vertex algebra ideal.\end{lemma}

\section{The $\cW_{1+\infty}$ algebra}
Let $\cD$ be the Lie algebra of regular differential operators on $\mathbb{C}\setminus \{0\}$, with coordinate $t$. 
A standard basis for $\cD$ is $$J^l_k = -t^{l+k} (\partial_t)^l, \qquad k\in \mathbb{Z}, \qquad l\in \mathbb{Z}_{\geq 0},$$ where 
$\partial_t = \frac{d}{dt}$. An alternative basis is $\{t^k D^l|\ k\in \mathbb{Z},\ l\in \mathbb{Z}_{\geq 0}\}$, where $D = t\partial_t$. There is a 2-cocycle on $\cD$ given by 
\begin{equation}\label{cocycle} 
\Psi\bigg(f(t) (\partial_t)^m,  g(t) (\partial_t)^n\bigg) = \frac{m! n!}{(m+n+1)!} \text{Res}_{t=0} f^{(n+1)}(t) g^{(m)}(t) dt,
\end{equation} 
and a corresponding central extension $\hat{\cD} = \cD \oplus \mathbb{C} \kappa$, which was first studied by Kac and Peterson in \cite{KP}. 
$\hat{\cD}$ has a $\mathbb{Z}$-grading $\hat{\cD} = \bigoplus_{j\in\mathbb{Z}} \hat{\cD}_j$ by weight, given by
$$\text{wt} (J^l_k) = k, \qquad \text{wt} (\kappa) = 0,$$ and a triangular decomposition 
$\hat{\cD} = \hat{\cD}_+\oplus\hat{\cD}_0\oplus \hat{\cD}_-$, where $\hat{\cD}_{\pm} = \bigoplus_{j\in \pm \mathbb{N}} \hat{\cD}_j$ and $\hat{\cD}_0 = \cD_0\oplus \mathbb{C}\kappa.$

Let $\cP$ be the parabolic subalgebra of $\cD$ consisting of differential operators which extend to all of $\mathbb{C}$, which has a basis $\{J^l_k|\ l\geq 0,\ l+k\geq 0\}$. The cocycle $\Psi$ vanishes on $\cP$, 
so $\cP$ may be regarded as a subalgebra of $\hat{\cD}$, and $\hat{\cD}_0\oplus \hat{\cD}_+\subset \hat{\cP}$, where 
$\hat{\cP} = \cP\oplus \mathbb{C}\kappa$. Given $c\in \mathbb{C}$, let $\C_c$ denote the one-dimensional $\hat{\cP}$-module on which $\kappa$ acts by $c\cdot \text{id}$ and $J^{l}_k$ acts by zero. The induced $\hat{\cD}$-module $$\cM_c = U(\hat{\cD})\otimes_{U(\hat{\cP})} \C_c$$ is known as the {\it vacuum $\hat{\cD}$-module of central charge $c$}. $\cM_c$ has a vertex algebra structure and is generated by fields $$J^l(z) = \sum_{k\in\mathbb{Z}} J^l_k z^{-k-l-1}, \qquad l\geq 0$$ of weight $l+1$. The modes $J^l_k$ represent $\hat{\cD}$ on $\cM_c$, and we rewrite these fields in 
the form
$$J^l(z) = \sum_{k\in\mathbb{Z}} J^l(k) z^{-k-1},$$ 
where $J^l(k) = J^l_{k-l}$. In fact, $\cM_c$ is {\it freely} generated by $\{J^l(z)|~l\geq 0\}$; the set of iterated Wick products 
$$ :\partial^{i_1}J^{l_1}(z)\cdots \partial^{i_r} J^{l_r}(z):,$$ such that $l_1\leq \cdots \leq l_r$ and $i_a\leq i_b$ if $l_a = l_b$, forms a basis for $\cM_c$. 

A weight-homogeneous element $\omega\in \mathcal{M}_c$ is called a {\it singular vector} if $J^l\circ_k \omega = 0$ for all $k>l\geq 0$. The maximal proper $\hat{\mathcal{D}}$-submodule $\mathcal{I}_c$ is the vertex algebra ideal generated by all singular vectors $\omega\neq 1$, and the unique irreducible quotient $\mathcal{M}_c/\mathcal{I}_c$ is denoted by $\mathcal{W}_{1+\infty,c}$. We denote the image of $J^l$ in $\cW_{1+\infty,c}$ by $j^l$. The cocycle \eqref{cocycle} is normalized so that $\cM_c$ is reducible if and only if $c\in \mathbb{Z}$. For each integer $n \geq 1$, $\cI_n$ is generated by a singular vector of weight $n+1$, and $\cW_{1+\infty,n}\cong \mathcal{W}(\mathfrak{g}\mathfrak{l}(n))$, and in particular has a minimal strong generating set $\{j^0, j^1,\dots, j^{n-1}\}$ \cite{FKRW}. Similarly, it was shown in \cite{LI} that $\cI_{-n}$ is generated by a singular vector of weight $(n+1)^2$, and $\cW_{1+\infty,-n}$ has a minimal strong generating set $\{j^0, j^1,\dots, j^{n^2+2n-1}\}$. It is known \cite{W} that $\cW_{1+\infty,-1}\cong \cW(\gg\gl(3))$, but no identification of $\cW_{1+\infty,-n}$ with a standard $\cW$-algebra is known for $n>1$.

\section{The super $\cW_{1+\infty}$ algebra}
Following the notation in \cite{CW}, we denote by $\cS\cD$ the Lie superalgebra of regular differential operators on the super circle $S^{1|1}$. A standard basis for $\cS\cD$ is given by 
$$ t^{k+1} (\partial_t)^l \theta \partial_{\theta}, \qquad t^{k+1} (\partial_t)^l \partial_{\theta} \theta, \qquad t^{k+1} (\partial_t)^l \theta, \qquad t^{k+1} (\partial_t)^l \partial_{\theta},$$ for $l\in \mathbb{Z}_{\geq 0}$ and $k\in \mathbb{Z}$. Here $\theta$ is an odd indeterminate which commutes with $t$. The odd elements $\theta$ and $\partial_{\theta}$ generate a four-dimensional Clifford algebra $Cl$ with relation $\theta \partial_{\theta} + \partial_{\theta} \theta = 1$, and $\cS\cD = \cD \otimes Cl$.

Let $M(1,1)$ be the set of $2\times 2$ matrices of the form 
$$\left( \begin{matrix}  \alpha^0 & \alpha^+ \\ \alpha^-& \alpha^1  \end{matrix}\right),$$ 
where $\alpha^a \in \mathbb{C}$ for $a = 0,1,\pm$. There is a natural $\mathbb{Z}_2$-gradation on $M(1,1)$ where we define
$M_0 = \left(\begin{matrix} 1 & 0 \\ 0& 0  \end{matrix}\right)$ and $M_1= \left( \begin{matrix} 0 & 0 \\ 0& 1 \end{matrix}\right)$ to be even, and 
$M_+ = \left(\begin{matrix} 0 & 1 \\ 0& 0 \end{matrix}\right)$ and $M_- = \left(\begin{matrix} 0 & 0 \\ 1 & 0 \end{matrix}\right)$ to be odd. 
The supertrace $\text{Str}$ of the above matrix is $\alpha^0 - \alpha^1$. We have an isomorphism $Cl\cong M(1,1)$ of associative 
superalgebras given by 
$$M_0 \mapsto \partial_{\theta} \theta, \qquad M_1 \mapsto \theta \partial_{\theta}, \qquad M_+ \mapsto \partial_{\theta}, \qquad M_- \mapsto \theta.$$ 
Therefore we can regard $\cS\cD$ as the superalgebra of $2\times 2$ matrices with coefficients in $\cD$. Let $F(D)$ denote the matrix $$\left( \begin{matrix} f_0(D) & f_+(D) \\ f_-(D) & f_1(D)\end{matrix} \right), \qquad D = t\partial_t, \qquad f_a(x) \in \mathbb{C}[x],$$ which we regard as an element of $\cS\cD$. Define a $2$-cocycle $\Psi$ on $\cS\cD$ by
 \begin{equation}\label{supercocycle} 
\Psi(t^r F(D), t^s G(D)) = \left\{ \begin{matrix} \sum_{-r\leq j \leq -1} \text{Str} (F(j)G(j+r))   & r=-s \geq 0 \\ 0 & r+s \neq 0.  \end{matrix} \right .
\end{equation} 
We obtain a one-dimensional central extension $\widehat{\cS\cD} = \cS\cD \oplus \mathbb{C} C$, 
with bracket
 $$ [t^r F(D), t^s G(D)] = t^{r+s} (F(D+s)G(D) - (-1)^{|F||G|} F(D) G(D+r)) + \Psi(t^r F(D), t^s G(D))C.$$ 
Here $|\cdot |$ denotes the $\mathbb{Z}_2$-gradation. 
The principal $\mathbb{Z}$-gradation on $\widehat{\cS\cD}$ is given by 
\begin{equation} \begin{split} & \text{wt}(C) = 0, \qquad  \text{wt} (t^n f(D) \partial_{\theta} \theta )= \text{wt} (t^n f(D)  \theta\partial_{\theta} ) = n, \\ & \text{wt} (t^{n+1} f(D) \partial_{\theta}) = \text{wt} (t^n f(D)  \theta ) = n+\frac{1}{2}.\end{split}\end{equation}
This defines the triangular decomposition 
$$\widehat{\cS\cD} = \widehat{\cS\cD}_- \oplus \widehat{\cS\cD}_0 \oplus \widehat{\cS\cD}_+, \qquad \widehat{\cS\cD}_{\pm} = \bigoplus_{j\in\pm \mathbb{N}/2} \widehat{\cS\cD}_j.$$ 
Define $J^{a,k}_n = J^{k}_n M_a$ for $a = 0,1,\pm$, and define the parabolic subalgebra $\cS\cP \subset \cS\cD$ to be the Lie algebra spanned by 
$$\{J^{a,k}_n| k+n \geq 0,\ n\in \mathbb{Z},\ k\in \mathbb{Z}_{\geq 0},\ a = 0,1,\pm \}.$$ 
The cocyle \eqref{supercocycle} vanishes on $\cS\cP$, so $\cS\cP$ is a subalgebra of $\widehat{\cS\cD}$, and 
$\widehat{\cS\cD}_0\oplus \widehat{\cS\cD}_+\subset \widehat{\cS\cP}$, where $\widehat{\cS\cP} = \cS\cP \oplus \mathbb{C}C$. 
 
Given $c\in \mathbb{C}$, let $\C_c$ denote the one-dimensional $\widehat{\cS\cP}$-module on which $C$ acts by $c\cdot \text{id}$ and $J^{a,k}_n$ acts by zero. The induced $\widehat{\cS\cD}$-module
$$\cM_c(\widehat{\cS\cD}) = U(\widehat{\cS\cD}) \otimes_{U(\widehat{\cS\cP})} \C_c,$$ 
is known as the {\it vacuum $\widehat{\cS\cD}$-module of central charge $c$}. 
\begin{prop}
For all $c\in \mathbb{C}$,
$$\cM_c(\widehat{\cS\cD}) \cong \cM_{-c}(\widehat{\cS\cD}).$$
\end{prop}
\begin{proof}
We will construct an automorphism of $\widehat{\cS\cD}$ that maps $C$ to $-C$ and hence $\cM_c(\widehat{\cS\cD})$
carries also an action of $\widehat{\cS\cD}$ at central charge $-c$, establishing the isomorphism.

Define the map $\Pi$ on $\cS\cD$ by
$$ \Pi(F(D)) =\left( \begin{matrix} f_1(D) & f_-(D) \\ f_+(D) & f_0(D)\end{matrix} \right).$$
This map respects the graded commutator of $2\times 2$ matrices and $\Pi\circ\Pi$ acts as the identity, hence $\Pi$ defines an automorphism on $\cS\cD$. Note that $\Pi$ does not respect the supertrace but changes sign, so the cocycle also satisfies
$$\Psi(t^r \Pi(F(D)), t^s \Pi(G(D)))= -\Psi(t^r F(D), t^s G(D)).$$
Defining $\Pi(C)=-C$ extends $\Pi$ to an automorphism of $\widehat{\cS\cD}$. 
\end{proof}
The module $\cM_c(\widehat{\cS\cD})$ possesses a vertex superalgebra structure, and is freely generated by fields 
$$J^{a,k} (z) = \sum_{n\in\mathbb{Z}} J^{a,k}_n z^{-n-k-1}, \qquad k\geq 0, \qquad a = 0,1,\pm.$$
Here $J^{0,k},J^{1,k}$ are even and have weight $k+1$, and $J^{+,k}, J^{-,k}$ are odd and have weights $k+1/2$, $k+3/2$, respectively. The modes $J^{a,k}_n$ represent $\widehat{\cS\cD}$ on $\cM_c(\widehat{\cS\cD})$, and we rewrite these fields in the form
$$J^{a,k}(z) = \sum_{k\in\mathbb{Z}} J^{a,k}(n) z^{-n-1}, \qquad J^{a,k}(n) = J^{a,k}_{n-k}.$$ 

Define a filtration $$(\cM_c(\widehat{\cS\cD}))_{(0)} \subset (\cM_c(\widehat{\cS\cD}))_{(1)}\subset \cdots$$ on $\cM_c(\widehat{\cS\cD})$ as follows: for $l\geq 0$, $(\cM_c(\widehat{\cS\cD}))_{(2l)}$ is the span of iterated Wick products of the generators $J^{a,k}$ and their derivatives of length at most $l$, and $(\cM_c(\widehat{\cS\cD}))_{(2l+1)} = (\cM_c(\widehat{\cS\cD}))_{(2l)}$. In particular, $J^{a,k}$ and its derivatives have degree $2$. Equipped with this filtration, $\cM_c(\widehat{\cS\cD})$ lies in the category $\mathcal{R}$, and $\text{gr}(\cM_c(\widehat{\cS\cD}))$ is the polynomial superalgebra $\mathbb{C}[\partial^l J^{a,k}|~ l,k\geq 0]$. Each element $J^{a,k}(m) \in \widehat{\cS\cP}$ for $k,m\geq 0$ gives rise to a derivation of degree zero on $\text{gr}(\cM_c(\widehat{\cS\cD}))$, and this action of $\widehat{\cS\cP}$ on $\text{gr}(\cM_c(\widehat{\cS\cD}))$ is independent of $c$.

There are some substructures of $\cM_c(\widehat{\cS\cD})$ that will be important to us. First, the fields $J^{0,0}, J^{1,0}, J^{+,0},J^{-,0}$ satisfy
\begin{equation}\label{gl11struc}
\begin{split} 
& J^{0,0}(z) J^{0,0}(w)  \sim c (z-w)^{-2},\qquad J^{1,0}(z) J^{1,0}(w) \sim -c (z-w)^{-2},\\ & J^{0,0}(z) J^{-,0}(w) \sim J^{-,0}(w)(z-w)^{-1},\qquad J^{0,0}(z) J^{+,0}(w) \sim -J^{+,0}(w)(z-w)^{-1}, \\ & J^{1,0}(z) J^{-,0}(w) \sim - J^{-,0}(w)(z-w)^{-1},\qquad J^{1,0}(z) J^{+,0}(w) \sim J^{+,0}(w)(z-w)^{-1},\\ & J^{+,0}(z) J^{-,0}(w) \sim c (z-w)^{-2} - (J^{0,0}(w)+J^{1,0}(w))(z-w)^{-1},\end{split}\end{equation} so they generate a copy of the affine vertex superalgebra associated to $\gg\gl(1|1)$ at level $c$. Next, recall that the $N=2$ superconformal vertex algebra $\cA_c$ of central charge $c$ is generated by fields $F, L, G^{\pm}$, where $L$ is  a Virasoro element of central charge $c$, $F$ is an even primary of weight one, and $G^{\pm}$ are odd primaries of weight $\frac{3}{2}$. These fields satisfy
\begin{equation}\label{n=2svir}
\begin{split}
&F(z) F(w) \sim \frac{c}{3} (z-w)^{-2},\qquad G^{\pm}(z) G^{\pm}(w)\sim 0,\\ 
& F(z) G^{\pm} (w) \sim \pm G^{\pm} (w) (z-w)^{-1},\\ 
& G^+(z) G^-(w) \sim \frac{c}{3} (z-w)^{-3} + F(w) (z-w)^{-2} + (L(w) + \frac{1}{2} \partial F(w) )(z-w)^{-1}.
\end{split}\end{equation}

We have a vertex algebra homomorphism $\cA_{c} \ra \cM_c(\widehat{\cS\cD})$ given by 
\begin{equation} \begin{split} & F \mapsto \frac{2}{3}J^{0,0} - \frac{1}{3} J^{1,0}, \qquad L \mapsto J^{0,1} + J^{1,1} - \frac{2}{3} \partial J^{0,0} - \frac{1}{6} \partial J^{1,1}, \\ & G^+ \mapsto J^{-,0}, \qquad G^- \mapsto -J^{+,1} + \frac{1}{3} \partial J^{+,0}.\end{split} \end{equation}

Finally, $\{J^{0,k}|\ k\geq 0\}$ and $\{J^{1,k}|\ k\geq 0\}$ generate copies of $\cM_c$ and $\cM_{-c}$ inside $\cM_c(\widehat{\cS\cD})$, respectively. 

\begin{lemma} $\cM_c$ and $\cM_{-c}$ form a Howe pair, i.e., a pair of mutual commutants, inside $\cM_c(\widehat{\cS\cD})$.
\end{lemma} 

\begin{proof} We show that $\text{Com} (\cM_{c}, \cM_c(\widehat{\cS\cD})) = \cM_{-c}$; the proof that $\text{Com} (\cM_{-c}, \cM_c(\widehat{\cS\cD})) = \cM_{c}$ is the same. Let $\omega \in \text{Com} (\cM_{c}, \cM_c(\widehat{\cS\cD}))$, and write $\omega$ as a sum of monomials 
\begin{equation}\label{itermon} \partial^{a_1} J^{0,i_1} \cdots \partial^{a_r} J^{0,i_r} \partial^{b_1} J^{1,j_1} \cdots \partial^{b_s} J^{1,j_s }\partial^{c_1} J^{+,k_1} \cdots \partial^{c_t} J^{+,k_t } \partial^{d_1} J^{-,l_1} \cdots \partial^{d_u} J^{-,l_u }.\end{equation}
Since $\omega$ commutes with $J^{0,0}$, we have $t=u$ for each such term. Suppose that $u>0$ for some such monomial, and let $l$ be the maximal integer such that $J^{-,l}$ appears in any such monomial. Since $J^{0,2} \circ_1 J^{-,l} = l J^{-,l+1}$, we would have $J^{0,2} \circ_1 \omega \neq 0$, so we conclude that $u=0$. Therefore $\omega \in \cM_c \otimes \cM_{-c}$, and since the center of $\cM_c$ is trivial, we conclude that $\omega \in \cM_{-c}$. \end{proof}

\begin{lemma} \label{weakfg} 
For each $c\in\mathbb{C}$, the sets $$S = \{J^{0,0},J^{1,0}, J^{+,0}, J^{-,0}, J^{0,1}\}, \qquad T = \{J^{0,0}, J^{1,0}, J^{+,0}, J^{-,0}, J^{+,1}, J^{-,1}\}$$ both generate $\cM(\widehat{\cS\cD})_c$ as a vertex algebra. \end{lemma}

\begin{proof} Let $\bra S \ket$ and $\bra T\ket$ denote the vertex subalgebras of $\cM(\widehat{\cS\cD})_c$ generated by $S$ and $T$, respectively. Note that $J^{+,0} \circ_0 J^{0,1} = J^{+,1}$ and $J^{-,0}\circ_0 J^{0,1} = -J^{-,1}$, so $J^{+,1}$ and $J^{-,1}$ both lie in $\bra S \ket$. Next, $J^{+,0} \circ_0 J^{-,1} = - J^{0,1} - J^{1,1}$, which shows that $J^{1,1} \in \bra S \ket$. So far, $J^{a,k} \in \bra S \ket$ for $a=0,1,\pm$ and $k=0,1$.
Next, we have 
$$J^{0,1} \circ_0 J^{-,1} = J^{-,2}, \qquad J^{1,1} \circ_0 J^{+,1} = J^{+,2},$$
$$J^{-,2} \circ_2 J^{+,2} - (J^{-,1} \circ_1 J^{+,2}) = -3 J^{1,2}, \qquad J^{-,2} \circ_2 J^{+,2} + 2 (J^{-,1} \circ_1 J^{+,2}) = -6 J^{0,2}.$$ This shows that $J^{a,k}\in \bra S \ket$ for $a=0,1,\pm$ and $k\leq 2$.

For $k\geq 1$, we have $$J^{0,2} \circ_1 J^{0,k-1} = (k+1) J^{0,k} -2 \partial J^{0,k-1}, \qquad J^{0,1}\circ_0 J^{0,k} = \partial J^{0,k}.$$ It follows that $\alpha\circ_1 J^{0,k-1} = (k+1) J^{0,k}$, 
where $\alpha = J^{0,2} -2 \partial J^{0,1}$. Since $\alpha\in \bra S \ket$, it follows by induction that $J^{0,k} \in \bra S \ket$ for all 
$k$. Next, we have $$J^{+,0} \circ_0 J^{0,k} = J^{+,k}, \qquad J^{-,0}\circ_0 J^{0,k} = -J^{-,k},$$ so $J^{+,k}$ and $J^{-,k}$ lie in 
$\bra S \ket$ for all $k$. Finally, $J^{+,0} \circ_0 J^{-,k} = - J^{0,k} - J^{1,k}$, which shows that $J^{1,k}$ lies in $\bra S \ket$ for 
all $k$. This shows that $\cM(\widehat{\cS\cD})_c = \bra S \ket$.

To prove that $\cM(\widehat{\cS\cD})_c = \bra T \ket$, it is enough to show that $J^{0,1} \in \bra T \ket$. First, we have 
$$\big(J^{-,1} \circ_0 J^{+,1}\big) \circ_1 J^{+,1} = -4 J^{+,2} + 2\partial J^{+,1},$$ 
which implies that $J^{+,2} \in \bra T \ket$. Finally, we have $$J^{-,0} \circ_1 J^{+,2} = -2 J^{0,1},$$ which shows that 
$J^{0,1} \in \bra T \ket$. \end{proof}

Lemma \ref{weakfg} shows that $\cM_c(\widehat{\cS\cD})$ is a finitely generated vertex algebra. However, $\cM_c(\widehat{\cS\cD})$ is not {\it strongly} generated by any finite set of vertex operators. This follows from the fact that $\text{gr}(\cM_c(\widehat{\cS\cD}))$ is the polynomial superalgebra with generators $\partial^l J^{a,k}$ for $k,l \geq 0$ and $a = 0,1,\pm$, which implies that there are no nontrivial normally ordered polynomial relations in $\cM_c(\widehat{\cS\cD})$. A weight-homogeneous element $\omega\in \cM_c(\widehat{\cS\cD})$ is called a {\it singular vector} if $\omega$ is annihilated by the operators $$J^{0,k}\circ_m, \qquad J^{1,k} \circ_m, \qquad J^{-,k} \circ_m, \qquad J^{+,k} \circ_r, \qquad  m>k, \qquad r\geq k.$$ The maximal proper $\widehat{\cS\cD}$-submodule $\cS\cI_c$ is the ideal generated by all singular vectors $\omega\neq 1$, and the super $\cW_{1+\infty}$ algebra $\cV_c(\widehat{\cS\cD})$ is the unique irreducible quotient $\cM_c(\widehat{\cS\cD})/\cS\cI_c$. We denote the projection $\cM_c(\widehat{\cS\cD})\ra \cV_c(\widehat{\cS\cD})$ by $\pi_{c}$, and we write 
\begin{equation}\label{lowercasej} j^{a,k} = \pi_c (J^{a,k}), \qquad k\geq 0.\end{equation} 
Clearly $\cV_c(\widehat{\cS\cD})$ is generated as a vertex algebra by the corresponding sets 
$$\{j^{0,0},j^{1,0}, j^{+,0}, j^{-,0}, j^{0,1}\}, \qquad \{j^{0,0},j^{1,0}, j^{+,0}, j^{-,0}, j^{+,1}, j^{-,1}\},$$ 
but there may now be normally ordered polynomial relations among $\{j^{a,k}| k\geq 0 \}$ and their derivatives. The actions of $V_c(\gg\gl (1|1))$ and the $N=2$ superconformal algebra $\cA_c$ on $\cM_c(\widehat{\cS\cD})$ descend to actions on $\cV_c(\widehat{\cS\cD})$ given by the same formulas, where $J^{a,k}$ is replaced by $j^{a,k}$. Likewise, $\{j^{0,k}|\ k\geq 0\}$ and $\{j^{1,k}|\ k\geq 0\}$ generate mutually commuting copies of $\cW_{1+\infty,c}$ and $\cW_{1+\infty,-c}$, respectively, inside $\cV_c(\widehat{\cS\cD})$.

\section{The case of positive integral central charge}
For $n\in \mathbb{Z}$, $\cM_n(\widehat{\cS\cD})$ is reducible and $\cV_n(\widehat{\cS\cD})$ has a nontrivial structure. For $n=0$, $J^{+,0}$ is a singular vector so $V_0(\widehat{\cS\cD}) \cong \mathbb{C}$, and since $\cV_n(\widehat{\cS\cD}) \cong \cV_{-n}(\widehat{\cS\cD})$ it suffices to consider the case $n\geq 1$. The starting point of our study is a free field realization of $\cV_{n}(\widehat{\cS\cD})$ as the $GL_n$-invariant subalgebra of the $bc\beta\gamma$-system $\cF$ of rank $n$ \cite{AFMO}. This indicates that the structure of $\cV_{n}(\widehat{\cS\cD})$ is deeply connected to classical invariant theory. 

The $\beta\gamma$-system $\cS$ was introduced in \cite{FMS}, and is the unique vertex algebra with even generators $\beta^{i}$, $\gamma^{i}$ for $i=1,\dots, n$, which satisfy the OPE relations \begin{equation}\label{betagamma} \begin{split} & \beta^i(z)\gamma^{j}(w)\sim \delta_{i,j} (z-w)^{-1},\qquad \gamma^{i}(z)\beta^j(w)\sim -\delta_{i,j} (z-w)^{-1}, \\ & \beta^i(z)\beta^j(w)\sim 0,\qquad \gamma^i(z)\gamma^j (w)\sim 0.\end{split} \end{equation}
There is a one-parameter family of conformal structures 
\begin{equation}\label{oneparameter} 
L^{\cS}_{\lambda} = \lambda \sum_{i=1}^n :\beta^i\partial\gamma^i: + (\lambda-1) \sum_{i=1}^n :\partial\beta^i\gamma^i:
\end{equation} 
of central charge $n(12\lambda^2 - 12\lambda + 2)$, under which $\beta^i$ and $\gamma^i$ are primary of conformal weights $\lambda$ and $1-\lambda$, respectively. Similarly, the $bc$-system $\cE$ is the unique vertex superalgebra with odd generators $b^{i}$, $c^{i}$ for $i=1,\dots, n$, which satisfy the OPE relations 
\begin{equation}\label{bcope} \begin{split} & b^i (z) c^{j}(w)\sim\delta_{i,j} (z-w)^{-1}, \qquad c^{i}(z) b^j(w)\sim \delta_{i,j} (z-w)^{-1},\\ & b^i (z)b^j(w)\sim 0, \qquad c^{i}(z)c^{j}(w)\sim 0.\end{split}\end{equation} 
There is a similar family of conformal structures 
\begin{equation}\label{bconeparameter} 
L^{\cE}_{\lambda} = (1-\lambda) \sum_{i=1}^n :\partial b^i c^i: - \lambda \sum_{i=1}^n  :b^i \partial c^i :
\end{equation} 
of central charge $n(-12\lambda^2 + 12 \lambda -2)$, under which $b^{i}$ and $c^{i}$ are primary of conformal weights $\lambda$ and $1-\lambda$, respectively. The $bc\beta\gamma$-system $\cF$ is just $\cE \otimes \cS$, and we will assign $\cF$ the conformal structure $$L^{\cF} = L^{\cS}_{5/6} + L^{\cE}_{1/3},$$ under which $\beta^i$, $\gamma^{i}$, $b^i$, $c^{i}$ have weights $5/6$, $1/6$, $1/3$, $2/3$, respectively.

$\cF$ admits a good increasing filtration
\begin{equation} \label{filtw} 
\cF_{(0)}\subset \cF_{(1)}\subset \cdots, \qquad \cF = \bigcup_{k\geq 0} \cF_{(k)},
\end{equation} 
where $\cF_{(k)}$ is spanned by iterated Wick products of the generators $b^i, c^{i}, \beta^i, \gamma^{i}$ and their derivatives, of length at most $k$. This filtration is $GL_n$-invariant, and we have an isomorphism 
of supercommutative rings 
\begin{equation}\label{assgrad} 
\text{gr}(\cF) \cong \text{Sym}(\bigoplus_{k\geq 0} (V_k \oplus V^*_k)) \bigotimes \bigwedge (\bigoplus_{k\geq 0} (U_k \oplus U^*_k)).
\end{equation} 
Here $V_k$, $U_k$ are copies of $\mathbb{C}^n$, and $V^*_k$, $U^*_k$ are copies of $(\mathbb{C}^n)^*$, as $GL_n$-modules. The generators of 
$\text{gr}(\cF)$ are $\beta^{i}_k$, $\gamma^{i}_k$, $b^{i}_k$, and $c^{i}_k$, which correspond to the vertex operators 
$\partial^k \beta^{i}$, $\partial^k \gamma^{i}$, $\partial^k b^{i}$, and $\partial^k c^{i}$, respectively for $k\geq 0$.

\begin{thm} (Awata-Fukuma-Matsuo-Odake) \label{TAFMO} There is an isomorphism $\cV_{n}(\widehat{\cS\cD})\rightarrow \cF^{GL_n}$ given by
\begin{equation} \label{bgrealizationi} 
\begin{split} 
j^{0,k} \mapsto  -\sum_{i=1}^n :b^i \partial^k c^i:,\qquad j^{1,k} \mapsto \sum_{i=1}^n :\beta^i \partial^k \gamma^i:, \\ j^{+,k} \mapsto - \sum_{i=1}^n :b^i \partial^k \gamma ^i:,\qquad j^{-,k} \mapsto  \sum_{i=1}^n :\beta^i \partial^k c^i:.\end{split}\end{equation} 
\end{thm}

The Virasoro element $L$ of $\cV_{n}(\widehat{\cS\cD})$ is given by \begin{equation} \label{virasoroleveln} L = j^{0,1} + j^{1,1} - \frac{2}{3} \partial j^{0,0} - \frac{1}{6} \partial j^{1,0}.\end{equation} Clearly $L$ maps to $L^{\cF}$, and the above map is a morphism in the category $\cR$. Note that the copies of $\cW_{1+\infty,n}$ and $\cW_{1+\infty,-n}$ generated by $\{j^{0,k}|\ k\geq 0\}$ and $\{j^{1,k}|\ k\geq 0\}$, respectively, form a Howe pair inside $\cV_{n}(\widehat{\cS\cD})$. This is clear from Theorem \ref{TAFMO}. Any $\omega \in \text{Com}(\cW_{1+\infty,n}, \cV_{n}(\widehat{\cS\cD}))$ must commute with $j^{0,1}$, so it cannot depend on $b^i, c^i$ and their derivatives. Similarly, any $\omega \in \text{Com}(\cW_{1+\infty,-n}, \cV_{n}(\widehat{\cS\cD}))$ must commute with $j^{1,1}$, so it cannot depend on $\beta^i, \gamma^i$ and their derivatives.

The identification $\cV_{n}(\widehat{\cS\cD}) \cong \cF^{GL_n}$ suggests an alternative strong generating set for $\cV_{n}(\widehat{\cS\cD})$ coming from classical invariant theory. Since $GL_n$ preserves the filtration on $\cF$, we have \begin{equation} \label{grisos} \text{gr}(\cV_{n}(\widehat{\cS\cD})) \cong \text{gr}(\cF^{GL_n}) \cong \text{gr}(\cF)^{GL_n}.\end{equation} The generators and relations for $\text{gr}(\cF)^{GL_n}$ are given by Weyl's first and second fundamental theorems of invariant theory for the standard representation of $GL_n$ \cite{We}. This theorem was originally stated for the $GL_n$-invariants in the symmetric algebra, but the following is an easy generalization to the case of odd as well as even variables.

\begin{thm}(Weyl) \label{weylfft} For $k\geq 0$, let $V_k$ and $U_k$ be the copies of the standard $GL_n$-module $\mathbb{C}^n$ with basis $x_{i,k}$ and $y_{i,k}$, for $i=1,\dots,n$, respectively. Let $V^*_k$ and $U^*_k$ be the copies of $(\mathbb{C}^n)^*$ with basis $x'_{i,k}$ and $y'_{i,k}$, respectively. The invariant ring $$R=\bigg(\big(\text{Sym} \bigoplus_{k\geq 0} (V_k\oplus V^*_k )\big) \bigotimes  \big( \bigwedge  \bigoplus_{k\geq 0} (U_k\oplus U^*_k)\big) \bigg) ^{GL_n}$$ is generated by the quadratics 

\begin{equation}\label{weylgenerators} \begin{split} q^0_{k,l} = \sum_{i=1}^n y_{i,k} y'_{i,l},\qquad q^1_{k,l} = \sum_{i=1}^n x_{i,k} x'_{i,l},\\ q^+_{k,l} = \sum_{i=1}^n y_{i,k} x'_{i,l},\qquad q^-_{k,l} = \sum_{i=1}^n x_{i,k} y'_{i,l}.\end{split}\end{equation} 
Let $Q^0_{k,l}, Q^{1}_{k,l}$ be even indeterminates and let $Q^{+}_{k,l}, Q^{-}_{k,l}$ be odd indeterminates for $k,l\geq 0$. The kernel $I_n$ of the homomorphism \begin{equation}\label{weylquot} \mathbb{C}[Q^a_{k,l}] \rightarrow R, \qquad Q^a_{k,l}\mapsto q^a_{k,l},\end{equation} is generated by homogeneous polynomials $d_{I,J}$ of degree $n+1$ in the variables $Q^{a}_{k,l}$. Here $I=(i_0,\dots, i_{n})$ and $J = (j_0,\dots, j_{n})$ are lists of nonnegtive integers, where $i_r$ corresponds to either $V_{i_r}$ or $U_{i_r}$, and $j_s$ corresponds to either $V^*_{j_s}$ or $U^*_{j_s}$. We call indices $i_r$ and $j_s$ {\it bosonic} if they correspond to $V_{i_r}$ and $V^*_{j_s}$, and {\it fermionic} if they correspond to $U_{i_r}$ and $U^*_{j_s}$, respectively. Bosonic indices appearing in either $I$ or $J$ must be distinct, but fermionic incides can be repeated. Finally, $d_{I,J}$ is uniquely characterized by the condition that it changes sign if bosonic indices in either $I$ or $J$ are permuted, and remains unchanged if fermionic indices are permuted. If all indices are bosonic, \begin{equation}\label{weylrel}  d_{I,J}= \det \left[\begin{matrix} Q^1_{i_0,j_0} & \cdots & Q^1_{i_0,j_n} \cr  \vdots  & & \vdots  \cr  Q^1_{i_n,j_0}  & \cdots & Q^1_{i_n,j_n} \end{matrix}\right].\end{equation} \end{thm}

Under the identification \eqref{grisos}, the generators $q^a_{k,l}$ correspond to strong generators \begin{equation}\label{newgenomega} \begin{split} \omega^0_{k,l} = \sum_{i=1}^n :\partial^k b^i \partial^l c^i: , \qquad \omega^1_{k,l} = \sum_{i=1}^n :\partial^k\beta^i \partial^l \gamma^i:,\\ \omega^+_{k,l} = \sum_{i=1}^n :\partial^k b^i \partial^l \gamma^i:, \qquad \omega^-_{k,l} = \sum_{i=1}^n :\partial^k \beta^i \partial^l c^i: \end{split} \end{equation} of $\cV_{-n}(\widehat{\cS\cD})$, satisfying $\phi_2(\omega_{a,b}) = q_{a,b}$. In this notation, we have \begin{equation} j^{0,k} = -\omega^0_{0,k}, \qquad j^{1,k} = \omega^1_{0,k}, \qquad j^{+,k} = -\omega^+_{0,k}, \qquad j^{-,k} = \omega^-_{0,k}, \qquad k \geq 0.\end{equation} 

For each $m\geq 0$, let $A^a_m$ denote the vector space with basis $\{\omega^a_{k,l}|\ k+l = m\}$. We have $\text{dim}(A^a_{m}) = m+1$, and $\text{dim} \big(A^a_{m} / \partial(A^a_{m-1})\big) = 1$. Hence $A^a_{m}$ has a decomposition \begin{equation}\label{decompofa} A^a_{m} = \partial (A^a_{m-1})\oplus \bra j^{a,m} \ket ,\end{equation} where $\bra j^{a,m}\ket$ is the linear span of $j^{a,m}$. Clearly $\{\partial^{l} j^{0,m}|~ 0\leq l\leq m\}$ is a basis of $A_{m}$, so for $k+l = m$, $\omega^a_{k,l}\in A_{m}$ can be expressed uniquely in the form \begin{equation}\label{lincomb} \omega^a_{k,l} =\sum_{i=0}^m \lambda_i \partial^{i}j^{a,m-i},\end{equation} for constants $\lambda_i$. Hence $\{\partial^k j^{a,m}|\ k,m \geq 0\}$ and $\{\omega^a_{k,m }|\ k,m \geq 0\}$ are related by a linear change of variables. Using (\ref{lincomb}), we can define an alternative strong generating set $\{\Omega^a_{k,l}| \ k,l\geq 0\}$ for $\cM_{n}(\widehat{\cS\cD})$ by the same formula: for $k+l=m$, $$\Omega^a_{k,l} =\sum_{i=0}^m \lambda_i \partial^{i} J^{a,m-i}.$$ Clearly $\pi_{n}(\Omega^a_{k,l}) = \omega^a_{k,l}$. 

\section{The structure of the ideal $\cS\cI_{n}$}

Recall that the projection $\pi_{n}: \cM_{n}(\widehat{\cS\cD}) \ra \cV_{n}(\widehat{\cS\cD})$ with kernel $\cS\cI_{n}$ is a morphism in the category $\cR$. Under the identifications $$\text{gr}(\cM_{n}(\widehat{\cS\cD}))\cong \mathbb{C}[Q^a_{k,l}], \qquad \text{gr}(\cV_{n}(\widehat{\cS\cD}))\cong \mathbb{C}[q^a_{k,l}]/I_n,$$ $\text{gr}(\pi_{n})$ is just the quotient map \eqref{weylquot}. 

\begin{lemma} \label{ddef} For each classical relation $d_{I,J}$ there exists a unique vertex operator \begin{equation} \label{ddef} D_{I,J}\in (\cM_{n}(\widehat{\cS\cD}))_{(2n+2)}\cap \cS\cI_{n}\end{equation} satisfying \begin{equation}\label{uniquedij} \phi_{2n+2}(D_{I,J}) = d_{I,J}.\end{equation} These elements generate $\cS\cI_{n}$ as a vertex algebra ideal. \end{lemma} 

\begin{proof}
Clearly $\pi_{n}$ maps each filtered piece $(\cM_{n}(\widehat{\cS\cD}))_{(k)}$ onto $(\cV_{n}(\widehat{\cS\cD}))_{(k)}$, so the hypotheses of Lemma \ref{idealrecon} are satisfied. Since $I_{n} = \text{Ker} (\text{gr}(\pi_{n}))$ is generated by $\{d_{I,J}\}$, we can find $D_{I,J}\in (\cM_{n}(\widehat{\cS\cD}))_{(2n+2)}\cap \cS\cI_{n}$ satisfying $\phi_{2n+2}(D_{I,J}) = d_{I,J}$, such that $\{D_{I,J}\}$ generates $\cS\cI_{n}$. If $D'_{I,J}$ also satisfies \eqref{uniquedij}, we would have $D_{I,J} - D'_{I,J}\in (\cM_{n}(\widehat{\cS\cD}))_{(2n)} \cap \cS\cI_{n}$. Since there are no relations in $\cV_{n}(\widehat{\cS\cD})$ of degree less than $2n+2$, we have $D_{I,J} - D'_{I,J}=0$. \end{proof}

Recall the generators $b^i_j, c^i_j, \beta^i_j,\gamma^i_j$ of $\text{gr}(\cF)$ corresponding to $\partial^j b^i, \partial^j c^i, \partial^j\beta^i, \partial^j\gamma^i$. Let $W\subset \text{gr}(\cF)$ be the vector space with basis $\{b^i_j, c^i_j,\beta^i_j, \gamma^i_j|\ j\geq 0\}$, and for each $m\geq 0$, let $W_m$ be the subspace with basis $\{b^i_j, c^i_j, \beta^i_j, \gamma^i_j |\ 0\leq j\leq m\}$. Let $\phi:W\ra W$ be a linear map of weight $w\geq 1$, such that \begin{equation}\label{arbmap} \phi(b^i_j) = \lambda^b_j b^i_{j+w},\qquad \phi(c^i_j) = \mu^c_j c^i_{j+w}, \qquad \phi(\beta^i_j) = \lambda^{\beta}_j \beta^i_{j+w},\qquad \phi(\gamma^i_j) = \mu^{\gamma}_j \gamma^i_{j+w}\end{equation} for constants $\lambda^{b}_j,\mu^{c}_j,\lambda^{\beta}_j,\mu^{\gamma}_j \in \mathbb{C}$ which are independent of $i$. For example, the restrictions of $j^{0,k}(k-w)$ and $j^{1,k}(k-w)$ to $W$ is such a map for $k\geq w$.

\begin{lemma} \label{third} Fix $w\geq 1$ and $m\geq 0$, and let $\phi$ be a linear map satisfying (\ref{arbmap}). Then the restriction $\phi \big|_{W_m}$ can be expressed uniquely as a linear combination of the operators $$\{j^{0,k}(k-w)\big|_{W_m}, \qquad j^{1,k}(k-w)\big|_{W_m}|\ 0\leq k-w \leq 2m+1\}.$$\end{lemma}

\begin{proof} The argument is the same as the proof of Lemma 6 of \cite{LII}.\end{proof}

\begin{lemma} \label{fourth} Fix $w\geq 1$ and $m\geq 0$, and let $\phi$ be a linear map satisfying 
\begin{equation}\label{arbmapi} \phi(b^i_j) = \lambda_j \beta^i_{j+w},\qquad \phi(c^i_j) =0, \qquad \phi(\beta^i_j) = 0,\qquad \phi(\gamma^i_j) = 0.\end{equation} for constants $\lambda_j \in \mathbb{C}$ which are independent of $i$. Then the restriction $\phi \big|_{W_m}$ can be expressed as a linear combination of $j^{-,k}(k-w)$ for $0\leq k-w \leq 2m+1$.

Similarly, let $\psi$ be a linear map satisfying
\begin{equation}\label{arbmapii} \psi(b^i_j) = 0,\qquad \psi(c^i_j) = \mu_j \gamma^i_{j+w}, \qquad \psi(\beta^i_j) = 0 ,\qquad \psi(\gamma^i_j) = 0.\end{equation} for constants $\mu_j \in \mathbb{C}$. Then the restriction $\psi \big|_{W_m}$ can be expressed as a linear combination of the operators $j^{+,k}(k-w) \big|_{W_m}$ for $0\leq k-w \leq 2m+1$.\end{lemma}

\begin{proof} This is easy to extract from Lemma \ref{third} using the $\gg\gl(1|1)$ structure. 
\end{proof}

Let $\bra D_{I,J} \ket$ denote the vector space with basis $\{D_{I,J}\}$ where $I,J$ are as in Theorem \ref{weylfft}. We have $\bra D_{I,J} \ket = (\cM_{n}(\widehat{\cS\cD}))_{(2n+2)}\cap \cS\cI_{n}$, and clearly $\bra D_{I,J} \ket$ is a module over the Lie algebra $\widehat{\cS\cP} \subset \widehat{\cS\cD}$ generated by $\{J^{a,k}(m) |~ m,k \geq 0\}$, since $\widehat{\cS\cP}$ preserves both the filtration on $\cM_{n}(\widehat{\cS\cD})$ and the ideal $\cS\cI_{n}$. The action of $\widehat{\cS\cP}$ on $\bra D_{I,J} \ket$ is by \lq\lq weighted derivation" in the following sense. Given $I = (i_0,\dots,i_{n})$, $J= (j_0,\dots,j_{n})$ and given an even operator $\phi \in \widehat{\cS\cP}$ satisfying \eqref{arbmap}, we have 
\begin{equation} \label{weightederivation} \phi(D_{I,J}) = \sum_{r=0}^{n} \lambda_{i_r} D_{I^r,J} + \mu_{j_r} D_{I,J^r},\end{equation} for lists $I^r = (i_0,\dots, i_r + w ,\dots, i_{n})$ and $J^r = (j_0,\dots,  j_r + w,\dots, j_{n})$. Here $\lambda_{i_r} = \lambda^b_{i_r}$ if $i_r$ is fermionic, and $\lambda_{i_r} = \lambda^{\beta}_{i_r}$ if $i_r$ is bosonic. Moreover, $i_r+w$ has the same parity as $i_r$, i.e., it is bosonic (respectively fermionic) if and only if $i_r$ is. Similarly, $\mu_{j_r} = \mu^c_{j_r}$ if $j_r$ is fermionic, and $\mu_{j_r} = \mu^{\gamma}_{j_r}$ if $j_r$ is bosonic, and the parities of $j_r$ and $j_r+w$ are the same. The odd operators $\phi \in \widehat{\cS\cP}$ given by Lemma \ref{fourth} have a similar derivation property except that they reverse the parity of the entries $i_r$ and $j_r$.

For each $n\geq 1$, there are four distinguished elements in $\bra D_{I,J} \ket$, which correspond to $I = (0,\dots, 0) = J$. Define $D_+$ to be the element where all entries of $I$ are fermionic, and one entry $J$ is bosonic. Similarly, define $D_-$ to be the element where one entry of $I$ is bosonic and all entries of $J$ are fermionic. Finally, define $D_0$ to be the element where all entries in both $I$ and $J$ are fermionic, and define $D_1$ to be the element where one entry of $I$ and one entry of $J$ are bosonic. Clearly $D_0, D_1, D_+, D_-$ have weights $n+1$, $n+1$, $n+1/2$, and $n+3/2$, respectively. It is clear that $D_+$ is the unique element of $\cS\cI_{n}$ of minimal weight $n+1/2$, and hence is a singular vector in $\cM_{n}(\widehat{\cS\cD})$. 

\begin{thm} \label{uniquesv} $D_+$ generates $\cS\cI_{n}$ as a vertex algebra ideal.\end{thm} 
\begin{proof}
Since $\cS\cI_{n}$ is generated by $\bra D_{I,J} \ket$ as a vertex algebra ideal, it suffices to show that $\bra D_{I,J} \ket$ is generated by $D_+$ as a module over $\widehat{\cS\cP}$. Let $\cS\cI'_{n}$ denote the ideal in $\cM_n(\widehat{\cS\cD})$ generated by $D_+$, and let $\bra D_{I,J} \ket^{(m)}$ denote the subspace spanned by elements $D_{I,J}$ with $|I| + |J| = m$. We will prove by induction on $m$ that $\bra D_{I,J} \ket^{(m)} \subset \cS\cI'_n$. 

First we need to show that $\bra D_{I,J} \ket ^{(0)} \subset \cS\cI'_n$, i.e.,  $D_0$, $D_1$, and $D_-$ lie in $\cS\cI'_n$. Note that $J^{-,0} \circ_0 D_+ = D_0 + (n+1)D_1$, so $D_0 + (n+1)D_1$ lies in $\cS\cI'_n$. By Lemma \ref{third}, we can find $\phi\in \widehat{\cS\cP}$ such that \begin{equation} \begin{split} & \phi(\beta^i_0) = \beta^i_1, \qquad \phi(\beta^i_r) = 0,\qquad r>0, \\ & \phi(\gamma^i_s) = 0, \qquad \phi(c^i_s) = 0, \qquad \phi(b^i_s) = 0, \qquad s\geq 0.\end{split}\end{equation} We have $\phi(D_0) = 0$ and $\phi(D_1) = D_{I,J}$ where $I= (1,0,\dots, 0)$ and $J = (0,\dots, 0)$. Moreover, the entry $1$ in $I$ is bosonic, and all other entries of $I$ are fermionic, and $J$ contains one bosonic entry and $n-1$ fermionic entries. It follows that $\phi(D_0 + (n+1)D_1) = (n+1)D_{I,J}$, so $D_{I,J} \in \cS\cI'_n$. Next, note that $j^{1,1}(2)( \beta^i_1) = 2\beta^i_0$, so $J^{1,1} \circ_2(D_{I,J}) = 2 D_1$. This shows that $D_1 \in \cS\cI'_n$, so $D_0 \in \cS\cI'_n$ as well. Finally, $J^{-,0} \circ_0 (D_0) = D_-$, so $D_-$ also lies in $\cS\cI'_n$.

For $m>0$, we assume inductively that $\bra D_{I,J} \ket^{(r)}$ lies in $\cS\cI'_n$ for $0\leq r <m$. Fix $I = (i_0,\dots, i_{n})$ and $J = (j_0,\dots, j_n)$ with $|I|+|J| = m$. 

{\bf Case 1}: $I = (0,\dots, 0)$, and $j_0,\dots, j_n$ are all fermionic. Since $m>0$, at least one of the $j_k$'s is nonzero. Let $J'$ be obtained from $J$ by replacing $j_k$ with $0$. By Lemma \ref{third}, we can find $\phi \in \widehat{\cS\cP}$ with the property that \begin{equation} \begin{split} & \phi( c^i_0) =  c^i_{j_k}, \qquad \phi(c^i_r) = 0, \qquad r>0, \\& \phi(\gamma^i_s) = 0, \qquad \phi(b^i_s) =0, \qquad\phi(\beta^i_s) = 0, \qquad s\geq 0.\end{split}\end{equation} Then $\phi(D_{I,J'}) = \lambda D_{I,J}$ where $\lambda$ is a nonzero constant depending on the number of indices appearing in $J$ which are zero. Since $D_{I,J'} \in  \bra D_{I,J}\ket^{(m-j_k)} \subset \cS\cI'_n$, we have $D_{I,J} \in \cS\cI'_n$.

\smallskip

{\bf Case 2}: $I = (0,\dots, 0)$, and for some $0\leq r<n$, $j_0,\dots, j_r$ are fermionic and $j_{r+1},\dots, j_n$ are bosonic. If one of the fermionic entries $j_k\neq 0$ for $0\leq k \leq r$, we proceed as in Case 1. If $j_0 = \cdots = j_r$, there exists $j_k>0$ for some $ k=r+1,\dots, n$. Let $J'$ be obtained from $J$ by replacing the {\it bosonic} entry $j_k$ with the {\it fermionic} entry $0$. Then $D_{I,J'}\in \bra D_{I,J}\ket^{(m-j_k)} \subset \cS\cI'_n$. Using Lemma \ref{fourth}, we can find $\phi \in  \widehat{\cS\cP}$ such that \begin{equation} \begin{split} &\phi( c^i_0) = \gamma^i_{j_k}, \qquad \phi(c^i_r)  =0, \qquad r>0,\\ & \phi(\gamma^i_s) = 0, \qquad \phi(b^i_s)  =0, \qquad \phi(\beta^i_s)  =0, \qquad s\geq 0.\end{split}\end{equation}
It follows that, up to a nonzero constant, $\phi(D_{I,J'}) = D_{I,J}$, so $D_{I,J} \in \cS\cI'_n$.

{\bf Case 3}: $I \neq (0,\dots,0)$. This is the same as Cases 1 and 2 with the roles of $I$ and $J$ reversed. \end{proof}

\section{A minimal strong finite generating set for $\cV_n(\widehat{\cS\cD})$}

Recall that $\{j^{0,k}|\ k\geq 0\}$ generates a copy of $\cW_{1+\infty,n}$ inside $\cV_n(\widehat{\cS\cD})$. It is well known \cite{FKRW} that the relation $D_0$ above is a singular vector for the action of the Lie subalgebra $\widehat{\cP} \subset \widehat{\cS\cP}$, and is of the form $$J^{0,n} - P(J^{0,0},\dots, J^{0,n-1}),$$ where $P$ is a normally ordered polynomial in $J^{0,0},\dots, J^{0,n-1}$ and their derivarives. Applying the projection $\pi_n:\cM_n(\widehat{\cS\cD}) \ra \cV_n(\widehat{\cS\cD})$ yields a decoupling relation
$$j^{0,n} = P(j^{0,0},\dots, j^{0,n-1}).$$ This relation is responsible for the isomorphism $\cW_{1+\infty,n} \cong \cW(\gg\gl(n))$. In fact, by applying the operator $j^{0,2} \circ_1$ repeatedly, it is easy to construct higher decoupling relations \begin{equation} \label{decoupfirst} j^{0,m} = P_m(j^{0,0}, j^{0,1},\dots, j^{0,n-1}),\qquad m\geq n.\end{equation} In particular, $\{j^{0,k}|\ 0\leq k<n\}$ strongly generates $\cW_{1+\infty,n}$. There are no nontrivial normally ordered polynomial relations among these generators and their derivatives, so they {\it freely} generate $\cW_{1+\infty,n}$.

\begin{thm} \label{sfgvsg} The set $\{ j^{0,k}, j^{1,k}, j^{+,k}, j^{-,k}|\ k=0,1,\dots, n-1\}$ is a minimal strong generating set for $\cV_n(\widehat{\cS\cD})$ as a vertex algebra.
\end{thm}

\begin{proof} We shall find all the necessary decoupling relations by acting on the relations \eqref{decoupfirst} by the copy of $\gg\gl(1|1)$ spanned by $j^{a,0}\circ_0$ for $a=0,1,\pm$. First, acting by $j^{+,0}\circ_0$ on the relations \eqref{decoupfirst} and using the fact that 
$$J^{+,0}\circ_0 \partial^m J^{0,k} =  \partial^mJ^{+,k},$$ 
we get relations
\begin{equation} \label{decoupsecond}
j^{+,m} = Q_m(j^{0,0}, j^{+,0}, j^{0,1},j^{+,1},\dots, j^{0,n-1} , j^{+,n-1}),
\end{equation} for $m\geq n$. Similarly, acting on \eqref{decoupfirst} by $j^{-,0}\circ_0$ and using
$$J^{-,0}\circ_0 \partial^m J^{0,k} =  -\partial^m J^{-,k},$$ 
we obtain decoupling relations
\begin{equation} \label{decoupthird}
j^{-,m} = R_m(j^{0,0}, j^{-,0}, j^{0,1},j^{-,1},\dots, j^{0,n-1} , j^{-,n-1}),
\end{equation} for $m\geq n$.
Finally, acting by $j^{+,0}\circ_0$ on \eqref{decoupthird} and using
$$J^{+,0}\circ_0\partial^m J^{-,k} =  -\partial^m J^{0,k} - \partial^m J^{1,k},$$
we obtain relations
\begin{equation} \label{decoupfourth} j^{0,m} + j^{1,m} = S_m(j^{0,0}, j^{1,0}, j^{+,0}, j^{-,0}, j^{0,1}, j^{1,1}, j^{+,1}, j^{-,1},\dots, j^{0,n-1}, j^{1,n-1}, j^{+,n-1}, j^{-,n-1}),\end{equation} for $m\geq n$. We can subtract from this the relation \eqref{decoupfirst}, obtaining 
\begin{equation} \label{decoupfifth}  
j^{1,m} = T_m(j^{0,0}, j^{1,0}, j^{+,0}, j^{-,0}, j^{0,1}, j^{1,1}, j^{+,1}, j^{-,1},\dots, j^{0,n-1}, j^{1,n-1}, j^{+,n-1}, j^{-,n-1}).
\end{equation}
The relations \eqref{decoupfirst}-\eqref{decoupthird} and \eqref{decoupfifth} imply that $\{ j^{0,k}, j^{1,k}, j^{+,k}, j^{-,k}|\ k=0,1,
\dots, n-1\}$ strongly generates $\cV_n(\widehat{\cS\cD})$. The fact that this set is {\it minimal} is a consequence of Weyl's second fundamental theorem of invariant theory for $GL_n$; there are no relations of weight less than or equal to $n+1/2$.
\end{proof}

\subsection{The cases $n=1$ and $n=2$}
It is immediate from Theorem \ref{sfgvsg} that $\cV_1(\widehat{\cS\cD}) \cong V_1(\gg\gl(1|1))$. In the case $n=2$, the decoupling relations for $j^{a,2}$ for $a=0,1,\pm$ are as follows:

 \begin{equation} \begin{split}
& j^{0,2} = -\frac{1}{6} :j^{0,0}j^{0,0}j^{0,0}: - \frac{1}{2} : j^{0,0} \partial j^{0,0}: + :j^{0,0}j^{0,1}: + \partial j^{0,1}  - 
\frac{1}{6} \partial^2 j^{0,0}, \\ 
& j^{+,2} = -\frac{1}{2} : j^{+,0} j^{0,0} j^{0,0}:  - \frac{1}{2} :j^{+,0} \partial j^{0,0}: + :j^{+,1} j^{0,0}: + 
 :j^{+,0} j^{0,1}: ,\\ 
 & j^{-,2} = -\frac{1}{2}  :j^{-,0} j^{0,0} j^{0,0}:  - \frac{1}{2}  :j^{-,0}  \partial j^{0,0}: - :\partial j^{-,0} j^{0,0}: + 
 : j^{-,1} j^{0,0}: \\
 & \qquad + :j^{-,0} j^{0,1}: -  \partial^2 j^{-,0}  + 2 \partial j^{-,1},\\
 & j^{1,2} = -:j^{-,0}j^{+,0}j^{0,0}: - \frac{1}{2} :j^{1,0}j^{0,0}j^{0,0}: - \frac{1}{3} :j^{0,0}j^{0,0}j^{0,0}: - 
 : \partial j^{-,0} j^{+,0}:   \\ 
& \qquad + :j^{-,1} j^{+,0}: + :j^{-,0}j^{+,1}:  - :\partial j^{1,0}j^{0,0}: - 
\frac{1}{2} :j^{1,0} \partial j^{0,0}: + :j^{1,1}j^{0,0}: \\
&\qquad  + :j^{1,0}j^{0,1}: - :j^{0,0}\partial j^{0,0}:  + 
 :j^{0,0}j^{0,1}: - \frac{1}{3} \partial^2 j^{0,0} + \partial j^{0,1 } - \partial^2 j^{1,0} + 2 \partial j^{1,1} .\end{split} \end{equation}

Since the original generating set $\{j^{a,k}|\ k\geq 0\}$ closes linearly under OPE, these decoupling relations allow us to write down all nonlinear OPE relations in $\cV_2(\widehat{\cS\cD})$ among the strong generating set $\{j^{a,k}|\ k=0,1\}$. For example, 
\begin{equation} j^{-,1}(z) j^{+,1}(w) \sim 2 (z-w)^{-4} + \big(j^{1,1}-j^{0,1}\big)(w)(z-w)^{-2} + \big(\partial j^{1,1}-j^{1,2}-j^{0,2}\big)(w)(z-w)^{-1},\end{equation} which yields
  \begin{equation} \begin{split} 
j^{-,1}(z) &j^{+,1}(w) \sim 2 (z-w)^{-4} + \big(j^{1,1}-j^{0,1}\big)(w)(z-w)^{-2} \\ 
& + \bigg(: j^{-,0} j^{+,0} j^{0,0}: + \frac{1}{2} :j^{1,0} j^{0,0} j^{0,0}: + \frac{1}{2} :j^{0,0} j^{0,0}j^{0,0}: + 
 :\partial j^{-,0} j^{+,0}:  - :j^{-,1} j^{+,0}: \\
 &- :j^{-,0} j^{+,1}:  +:\partial j^{1,0} j^{0,0}: + 
 \frac{1}{2} : j^{1,0} \partial j^{0,0}:  - :j^{1,1} j^{0,0}: - :j^{1,0} j^{0,1}: + 
 \frac{3}{2} :j^{0,0} \partial j^{0,0}:\\
&   - 2 : j^{0,0} j^{0,1}: + \frac{1}{2} \partial^2 j^{0,0} - 2 \partial j^{0,1}+ \partial^2 j^{1,0} - 
  \partial j^{1,1}  \bigg)(w)(z-w)^{-1}.\end{split} \end{equation}
 
Similarly, we have the following additional nonlinear OPEs: 
\begin{equation} 
\begin{split} 
j^{0,1}(z) j^{-,1}(w) &\sim j^{-,1}(w) (z-w)^{-2} + \bigg(-\frac{1}{2}  :j^{-,0} j^{0,0} j^{0,0}:  - \frac{1}{2}  :j^{-,0}  \partial j^{0,0}: - :\partial j^{-,0} j^{0,0}: \\ & + : j^{-,1} j^{0,0}:  + :j^{-,0} j^{0,1}: -  \partial^2 j^{-,0}  + 2 \partial j^{-,1}\bigg)(w)(z-w)^{-1},
\end{split} 
\end{equation}
\begin{equation}
\begin{split}
j^{1,1}(z) &j^{-,1}(w)  \sim j^{-,1}(w) (z-w)^{-2} +\bigg( \frac{1}{2}  :j^{-,0} j^{0,0} j^{0,0}:   + \frac{1}{2}  :j^{-,0}  \partial j^{0,0}:  \\& + :\partial j^{-,0} j^{0,0}:
 - : j^{-,1} j^{0,0}:  - :j^{-,0} j^{0,1}: +  \partial^2 j^{-,0}  -  \partial j^{-,1}\bigg)(w) (z-w)^{-1}.
\end{split} 
\end{equation}
 \begin{equation} 
\begin{split}
j^{0,1}(z) j^{+,1}(w) & \sim j^{+,1}(w) (z-w)^{-2} + \bigg( \frac{1}{2} : j^{+,0} j^{0,0} j^{0,0}: + \frac{1}{2} :j^{+,0} \partial j^{0,0}: \\ & - :j^{+,1} j^{0,0}: -
 :j^{+,0} j^{0,1}: +\partial j^{+,1}\bigg)(w)(z-w)^{-1}.
\end{split} 
\end{equation}
\begin{equation} 
\begin{split}j^{1,1}(z) j^{+,1}(w) &\sim j^{+,1}(w) (z-w)^{-2} + \bigg(-\frac{1}{2} : j^{+,0} j^{0,0} j^{0,0}:  - \frac{1}{2} :j^{+,0} \partial j^{0,0}: \\ &  + :j^{+,1} j^{0,0}: + 
 :j^{+,0} j^{0,1}:\bigg)(w) (z-w)^{-1}.
\end{split} 
\end{equation}
The remaining nontrivial OPEs in $\cV_2(\widehat{\cS\cD})$ are linear in the generators, and are omitted.

\section{A deformable family with limit $\cV_n(\widehat{\cS\cD})$} \label{sectiondefsv}
We will construct a deformable family of vertex algebras $\cB_{n,k}$ with the property that $ \cB_{n,\infty} = \lim_{k\ra \infty} \cB_{n,k} \cong \cV_n(\widehat{\cS\cD})$. The key property will be that for generic values of $k$, $\cB_{n,k}$ has a minimal strong generating set consisting of $4n$ fields, and has the same graded character as $\cV_n(\widehat{\cS\cD})$. 

First, we need to formalize what we mean by a deformable family. Let $K \subset \mathbb{C}$ be a subset which is at most countable, and let $F_K$ denote the $\mathbb{C}$-algebra of rational functions in a formal variable $\kappa$ of the form $\frac{p(\kappa)}{q(\kappa)}$ where $\text{deg}(p) \leq \text{deg}(q)$ and the roots of $q$ lie in $K$. A {\it deformable family} will be a free $F_K$-module $\cB$ with the structure of a vertex algebra with coefficients in $F_K$.  Vertex algebras over $F_K$ are defined in the same way as ordinary vertex algebras over $\mathbb{C}$. We assume that $\cB$ possesses a $\mathbb{Z}_{\geq 0}$-grading $\cB = \bigoplus_{m\geq 0} \cB[m]$ by conformal weight where each $\cB[m]$ is free $F_K$-module of finite rank. For $k\notin K$, we have a vertex algebra $$\cB_k = \cB / (\kappa - k),$$ where $(\kappa - k)$ is the ideal generated by $\kappa - k$. Clearly $\text{dim}_{\mathbb{C}}(\cB_k[m]) = \text{rank}_{F_K} (\cB[m])$ for all $k \notin K$ and $m\geq 0$. We have a vertex algebra $\cB_{\infty} = \lim_{\kappa\ra \infty} \cB$ with basis $\{\alpha_i|\ i\in I\}$, where $\{a_i|\ i \in I\}$ is any basis of $\cB$ over $F_K$, and $\alpha_i = \lim_{\kappa \ra \infty} a_i$. By construction, $\text{dim}_{\mathbb{C}}(\cB_{\infty}[m]) = \text{rank}_{F_K}(\cB[m])$ for all $m\geq 0$. The vertex algebra structure on $\cB_{\infty}$ is defined by \begin{equation} \alpha_i \circ_n \alpha_j = \lim_{\kappa \ra \infty} a_i \circ_n a_j, \qquad i,j\in I, \qquad n\in \mathbb{Z}. \end{equation} The $F_K$-linear map $\phi: \cB \ra \cB_{\infty}$ sending $a_i \mapsto \alpha_i$ satisfies \begin{equation} \label{preservecircle} \phi(\omega \circ_n \nu) = \phi(\omega) \circ_n \phi(\nu), \qquad \omega,\nu \in \cB, \qquad n\in \mathbb{Z}.\end{equation} Moreover, all normally ordered polynomial relations $P(\alpha_i)$ among the generators $\alpha_i$ and their derivatives are of the form $$\lim_{\kappa \ra \infty} \tilde{P}(a_i),$$ where $\tilde{P}(a_i)$ is a normally ordered polynomial relation among the $a_i$'s and their derivatives, which converges termwise to $P(\alpha_i)$. In other words, suppose that $$P(\alpha_i) = \sum_j c_j m_j(\alpha_i)$$ is a normally ordered relation of weight $d$, where the sum runs over all normally ordered monomials $m_j(\alpha_i)$ of weight $d$, and the coefficients $c_j$ lie in $\mathbb{C}$. Then there exists a relation $$\tilde{P}(a_i) = \sum_j c_j(\kappa) m_j(a_i)$$ where $\lim_{\kappa\ra \infty} c_j(\kappa) = c_j$ and $m_j(a_i)$ is obtained from $m_j(\alpha_i)$ by replacing $\alpha_i$ with $a_i$.

We are interested in the relationship between strong generating sets for $\cB_{\infty}$ and $\cB$. 

\begin{lemma} \label{passage} Let $\cB$ be a vertex algebra over $F_K$ as above. Let $U = \{\alpha_i|\ i\in I\}$ be a strong generating set for $\cB_{\infty}$, and let $T = \{a_i|\ i\in I\}$ be the corresponding subset of $\cB$, so that $\phi(a_i) = \alpha_i$. There exists a subset $S\subset \mathbb{C}$ containing $K$ which is at most countable, such that $F_S \otimes_{F_K}\cB$ is strongly generated by $T$. Here we have identified $T$ with the set $\{1 \otimes a_i|\ i\in I\} \subset F_S \otimes_{F_K} \cB$. \end{lemma}

\begin{proof} Without loss of generality, we may assume that $U$ is linearly independent. Complete $U$ to a basis $U'$ for $\cB_{\infty}$ containing finitely many elements in each weight, and let $T'$ be the corresponding basis of $\cB$ over $F_K$. Let $d$ be the first weight such that $U'$ contains elements which do not lie in $U$, and let $\{\alpha_{1,d},\dots, \alpha_{r,d}\}$ be the set of elements of $U'\setminus U$ of weight $d$. Since $U$ strongly generates $\cB_{\infty}$, we have decoupling relations in $\cB_{\infty}$ of the form $$\alpha_{j,d} = P_j(\alpha_i), \qquad j=1,\dots, r.$$ Here $P$ is a normally ordered polynomial in the generators $\{\alpha_i|\ i\in I\}$ and their derivatives. Let $a_{j,d}$ be the corresponding elements of $T'$. There exist relations $$a_{j,d} = \tilde{P}_j(a_i, a_{1,d},\dots, \widehat{a_{j,d}},\dots, a_{r,d}), \qquad j=1,\dots, r,$$ which converge termwise to $P_j(\alpha_i)$. Here $\tilde{P}_j$ does not depend on $a_{j,d}$ but may depend on $a_{k,d}$ for $k\neq j$. Since each $a_{k,d}$ has weight $d$ and $\tilde{P}_j$ is homogeneous of weight $d$, $\tilde{P}_j$ depends linearly on $a_{k,d}$. We can therefore rewrite these relations in the form $$\sum_{k=1}^r b_{jk} a_{k,d} = Q_j(a_i), \qquad b_{jk} \in F_K,$$ where $b_{jj} = 1$, $\lim_{\kappa\ra \infty} b_{jk} = 0$ for $j\neq k$, and $$Q_j(a_i) = \tilde{P}_j(a_i, a_{1,d},\dots, \widehat{a_{j,d}},\dots, a_{r,d})  + \sum_{k=1}^{j-1} b_{jk} a_{k,d}  +\sum_{k=j+1}^r b_{jk} a_{k,d}.$$Clearly $\lim_{\kappa \ra \infty} \det[b_{jk}] = 1$, so this matrix is invertible over the field of rational functions in $\kappa$. Let $S_d$ denote the union of $K$ with the set of distinct roots of the numerator of $\det[b_{jk}]$ regarded as a rational function of $\kappa$. We can solve this linear system over the ring $F_{S_d}$, so in $F_{S_d}\otimes_{F_K} \cB$ we obtain decoupling relations $$a_{j,d} = \tilde{Q}_j(a_i), \qquad j=1,\dots, r.$$ For each weight $d+1,d+2\dots$ we repeat this procedure, obtaining sets $$S_d\subset S_{d+1} \subset S_{d+2} \subset \cdots$$ and decoupling relations $$a = P(a_i)$$ in $F_{S_{d+i}}\otimes_{\mathbb{C}} \cB$, for each $a \in T'\setminus T$ of weight $d+i$. Letting $S  = \bigcup_{i\geq 0} S_{d+i}$, we obtain decoupling relations in $F_S \otimes_{F_K} \cB$ expressing each $a \in T'\setminus T$ as a normally ordered polynomial in $a_1,\dots,a_s$ and their derivatives. \end{proof}

\begin{cor} \label{passagecor} For $k\notin S$, the vertex algebra $\cB_k = \cB/ (\kappa -k)$ is strongly generated by the image of $T$ under the map $\cB \ra \cB_k$.
\end{cor}

Next we consider a deformation of $\cF^{GL_n}\cong \cV_n(\widehat{\cS\cD})$. Note that $\cF$ carries an action of $V_0(\gg\gl(n))$ which is just the sum of the action of $V_1(\gg\gl(n))$ on the $bc$-system $\cE$, and the action of $V_{-1}(\gg\gl(n))$ on the $\beta\gamma$-system $\cS$. Fix an orthonormal basis $\{\xi_i\}$ for $\gg\gl(n)$ relative to the normalized Killing form. We have the diagonal homomorphism \begin{equation} \label{diagonalhomomorphism} V_{k} (\gg\gl(n)) \ra V_k (\gg\gl(n)) \otimes \cF,\qquad \bar{X} ^{\xi_i} \mapsto \tilde{X}^{\xi_i}\otimes 1 + 1\otimes X^{\xi_i}.\end{equation} Here $\bar{X}^{\xi_i}$ and $\tilde{X}^{\xi_i}$ are the generators of the first and second copies of $V_{k}(\gg\gl(n))$, respectively, and $X^{\xi_i}$ are the generators of the image of $V_0(\gg\gl(n))$ inside $\cF$. Define $$\cB_{n,k} = \text{Com}(V_{k}(\gg\gl(n)), V_k(\gg\gl(n)) \otimes \cF).$$ There is a linear map $\cB_{n,k} \ra \cF^{GL_n}$ defined as follows. Each element $\omega \in \cB_{n,k}$  of weight $d$ can be written uniquely in the form $\omega = \sum_{r=0}^d \omega_r$ where $\omega_r$ lies in the space \begin{equation} \label{shapemon}(V_k(\gg\gl(n)) \otimes \cF)^{(r)}\end{equation} spanned by terms of the form $\alpha \otimes \nu$ where $\alpha \in V_k(\gg\gl(n))$ has weight $r$. Clearly $\omega_0 \in \cF^{GL_n}$ so we have a well-defined linear map \begin{equation} \label{limitmap} \phi_k: \cB_{n,k} \ra \cF^{GL_n}, \qquad \omega \mapsto \omega_0.\end{equation} Note that $\phi_k$ is not a vertex algebra homomorphism for any $k$.

\begin{lemma} \label{deformationofvgi} $\phi_k$ is injective whenever $V_k(\gg\gl(n))$ is a simple vertex algebra. \end{lemma}

\begin{proof}
Assume that $V_k(\gg\gl(n))$ is simple. Fix $\omega \in \cB_{n,k}$, and suppose that $\phi_k(\omega) = 0$. If $\omega \neq 0$, there is a minimal integer $r>0$ such that $\omega_r \neq 0$. We may express $\omega_r$ as a linear combination of terms of the form $\alpha \otimes \nu$ for which the $\nu$'s are linearly independent. Since $\omega$ lies in the commutant $\cB_{n,k}$, it follows that each of the above $\alpha$'s must be annihilated by $\tilde{X}^{\xi_i}(m)$ for all $m>0$. Since $\text{wt}(\alpha) = r>0$, this implies that $\alpha$ generates a nontrivial ideal in $V_k(\gg\gl(n))$, which is a contradiction. \end{proof}

Let $K\subset \mathbb{C}$ be the set of values of $k$ such that $V_k(\gg\gl(n))$ is {\it not} simple. This set is countable and is described explicitly by Theorem 0.2.1 in the paper \cite{GK} by Kac and Gorelik. As above, there exists a vertex algebra $\cB_n$ with coefficients in $F_K$ with the property that $\cB_n / (\kappa - k) = \cB_{n,k}$ for all $k\notin K$. The generators of $\cB_n$ are the same as the generators of $\cB_{n,k}$, where $k$ has been replaced by the formal variable $\kappa$, and the OPE relations are the same as well. The maps $\phi_k$ above give rise to a linear map $\phi_{\kappa}: \cB_n \ra F_K \otimes_{\mathbb{C}} \cF^{GL_n}$, which is not a vertex algebra homomorphism.

\begin{lemma} The map $\phi_{\kappa}: \cB_n \ra F_K \otimes_{\mathbb{C}} \cF^{GL_n}$ is a linear isomorphism, and the induced map $\phi =  \lim_{\kappa \ra \infty} \phi_{\kappa}$ is a vertex algebra isomorphism from $\cB_{n,\infty} \ra \cF^{GL_n}$. \end{lemma}

\begin{proof} Recall that by Lemma \ref{weakfg}, $\cF^{GL_n}$ is generated by $\{j^{0,0}, j^{1,0}, j^{+,0}, j^{-,0}, j^{0,1}\}$. To prove the surjectivity of $\phi_{\kappa}$, it is enough to find elements $\{t^{0,0}, t^{1,0}, t^{+,0}, t^{-,0}, t^{0,1}\}$ in $\cB_{n,k}$ such that $\phi_k(t^{a,k}) = j^{a,k}$ and $\lim_{k\ra \infty} t^{a,k} = j^{a,k}$. This is a straightforward calculation, where \begin{equation} \begin{split} \label{commcori} & t^{0,0} = 1\otimes j^{0,0} -\frac{1}{k}  \sum_{i} \tilde{X}^{\xi_i}  \otimes (X^{\xi_i} \circ_1 j^{0,0}),\qquad t^{1,0} = 1\otimes j^{1,0} -\frac{1}{k} \sum_{i} \tilde{X}^{\xi_i}  \otimes (X^{\xi_i} \circ_1 j^{1,0}),\\ & t^{\pm,0} =1\otimes  j^{\pm,0}, \qquad t^{0,1} = 1\otimes j^{0,1} - \frac{1}{k} \sum_{i} \tilde{X}^{\xi_i}  \otimes (X^{\xi_i} \circ_1 j^{0,1}).\end{split} \end{equation}
It is clear from \eqref{commcori} that $\phi$ is a vertex algebra isomorphism. \end{proof}

\begin{thm} Let $U$ be the strong generating set $\{ j^{0,l}, j^{1,l}, j^{+,l}, j^{-,l}|\ l=0,1,\dots, n-1\}$ for $\cV_n(\widehat{\cS\cD}) \cong \cF^{GL_n}$ given by Theorem \ref{sfgvsg}, and let $T =\{ t^{0,l}, t^{1,l}, t^{+,l}, t^{-,l}|\ l=0,1,\dots, n-1\}$ be the corresponding subset of $\cB_{n,k}$ with $\phi_k(t^{a,l}) = j^{a,l}$. For generic values of $k$, $T$ is a minimal strong generating set for $\cB_{n,k}$. \end{thm} 

\begin{proof} By Lemma \ref{passage} and Corollary \ref{passagecor}, $T$ strongly generates $\cB_{n,k}$ for generic $k$. If $T$ were not minimal, we would have a decoupling relation expressing $t^{a,l}$ as a normally ordered polynomial in the remaining elements of $T$ and their derivatives, for some $l\leq n-1$. This relation has weight at most $n+1/2$, and taking the limit as $k\ra \infty$ would give us a nontrivial relation in $\cV_n(\widehat{\cS\cD})$ of the same weight. But this is impossible by Theorem \ref {weylfft}, which implies that there are no such relations in $\cF^{GL_n}$.
\end{proof}

Finally, note that if $\{\nu_i|\ i\in I\}$ generates $\cF^{GL_n}$, not necessarily strongly, the corresponding subset $\{\omega_i |\ i\in I\}$ generates $\cB_{n,k}$ for generic values of $k$. This is immediate from the fact that if $\{\nu_i|\ i\in I\}$ generates $\cF^{GL_n}$, the set $$\{\nu_{i_1} \circ_{j_1}( \cdots (\nu_{i_r-1} \circ_{j_{r-1}} \nu_{i_r})\cdots )|\ i_1,\dots, i_r \in I,\ j_1,\dots, j_{r-1} \geq 0\}$$ strongly generates $\cF^{GL_n}$.

\begin{cor} \label{genbk} The sets $\{t^{0,0}, t^{1,0}, t^{+,0}, t^{-,0}, t^{0,1}\}$ and $\{t^{0,0}, t^{1,0}, t^{+,0}, t^{-,0}, t^{+,1}, t ^{-,1}\}$ both generate $\cB_{n,k}$ for generic values of $k$.
\end{cor}

\section{$\cW$-algebras of $\AKMSA{gl}{n}{n}$}  
 
As mentioned in the introduction, $\cW$-algebras can often be realized in various ways. In this section, we find a family of $\cW$-algebras $\cW_{n,k}$ associated to a certain simple and purely odd root system of $\SLSA{gl}{n}{n}$. We will see that $\cV_{n}(\widehat{\cS\cD}) = \lim_{k\ra \infty} \cW_{n,k}$. In the case $n=2$ we write down all nontrivial OPE relations in $\cW_{2,k}$ explicitly. 

\begin{defn}\label{def:glnn}
Let $X=\{E_{ij} | 1\leq i,j\leq 2n\}$ be the basis of a $\mathbb Z_2$-graded $\mathbb C$-vector space with gradation given by
\begin{equation}
|E_{ij}|\ = \ \begin{cases}
0 \quad &\text{for}\quad 1\leq i,j\leq n \quad \text{or}\quad n+1\leq i,j\leq 2n\\
1 \quad &\text{for}\quad 1\leq i\leq n ,\ \  n+1\leq j\leq 2n \quad \text{or}\quad 1\leq j\leq n ,\ \  n+1\leq i\leq 2n.
\end{cases}
\end{equation}
Then 
\begin{equation}
[E_{ij},E_{kl}]\ = \ \delta_{j,k}E_{il}-(-1)^{|E_{ij}||E_{kl}|}\delta_{i,l}E_{kj}
\end{equation}
provides the $\mathbb C$-span of $X$ with the structure of a Lie superalgebra, namely $\SLSA{gl}{n}{n}$.
A consistent, graded symmetric, invariant and nondegenerate bilinear form is given by
\begin{equation}
B(E_{ij},E_{kl})\ =\ \delta_{j,k}\delta_{i,l}\times \begin{cases} 1 \quad &\text{for}\quad 1\leq i\leq n\\ 
 -1 \quad &\text{for}\quad n+1\leq i\leq 2n.
 \end{cases}
\end{equation} 
\end{defn}
A Cartan subalgebra has a basis given by the $E_{ii}$. A root system is given by $\alpha_{ij}$ for $1\leq i\neq j\leq 2n$ with
$\alpha_{ij}(E_{kk})= \delta_{ik}-\delta_{jk}$. 
The parity of a root is defined as $|\alpha_{ij}|=|E_{ij}|$. The distinguished system of positive simple roots is given by $\alpha_{i, i+1}$, where only $\alpha_{n,n+1}$
is an odd root. We are interested in a system of positive simple and purely odd roots. Such a system is given by $\{\alpha_i=\alpha_{i,n+i},\beta_i=\alpha_{n+i,i+1}\}$.
Define $2n$ bosonic fields $J^\pm_i$, $1\leq i\leq n$, and $2n$ fermionic fields $\psi^\pm_i$, with operator products
\begin{equation}
J_i^\pm(z)J_j^\pm(w) \ \sim \ \pm \frac{k\delta_{i,j}}{(z-w)^2}, \qquad \qquad \psi_i^\pm(z)\psi_j^\pm(w)\ \sim \ \pm\frac{k\delta_{i,j}}{(z-w)},
\end{equation}
and all other operator products regular. The complex number $k$ will be called the level. 
Let $J_i^\pm(z)= \sum_{n\in \mathbb{Z}} J^\pm_i(n) z^{-n-1}$ be the expansion as a formal Laurent series of the field $J_i^\pm(z)$. The coefficients satisfy the commutation relations of a rank $2n$ Heisenberg algebra $[J^\pm_i(n),J^\pm_j(m)] = n \delta_{i.j}\delta_{n+m,0}k$
of level $k$. 
Let 
\[
 \phi^\pm_i(z) = q^\pm_i+ J^\pm_i(0) \ln z - \sum_{n\neq 0} \frac{J^\pm_i(n)}{n} x^{-n},
\]
 where $q^\pm_i$ satisfies $[J^\pm_j(n),q^\pm_i] = \pm\delta_{i,j}\delta_{n,0}k$ and we have $\partial \phi^\pm_i(z) = J^\pm_i(z)$,
so that 
\begin{equation}
\phi_i^\pm(z)\phi_j^\pm(w) \ \sim \ \pm k\delta_{i,j}\ln(z-w)\,.
\end{equation}
Let $\alpha=(\alpha^+_1,\dots,\alpha^+_n,\alpha^-_1,\dots,\alpha^-_n)\in\mathbb C^{2n}$, and consider the highest-weight module
$\mathcal H_\alpha$ of the Heisenberg vertex algebra of weight $\alpha$ with highest-weight vector $v_\alpha$
satisfying 
\[
 J^\pm_i(n)v_\alpha= \alpha^\pm_i\delta_{n,0}v_\alpha,\qquad n\geq 0.
\]
Denote by $e^+_i\in\mathbb C^{2n}$ (resp. $e^-_i\in\mathbb C^{2n}$) the vector 
with all entries zero except the one at the $i^{\text{th}}$ (resp. $(n+i)^{\text{th}}$) position being one. 
Then for $\eta\in\mathbb C$, the operator $e^{\eta q^\pm_i}$ acts as $e^{\eta q^\pm_i}(v_\alpha)=v_{\alpha+k\eta e^\pm_i}$, so it
maps $\mathcal H_\alpha\rightarrow \mathcal H_{\alpha+k\eta e^\pm_i}$. 
Define the field 
\[
 e^{\eta \phi^\pm_i(z)} = e^{\eta q^\pm_i} z^{\eta \alpha} \mbox{exp} \bigg(\eta \sum_{n>0} J^\pm_i(-n) \frac{z^n}{n}\bigg) \mbox{exp}\bigg(\eta\sum_{n<0} J^\pm_i(-n) \frac{z^n}{n}\bigg).
\]
 The $e^{\eta \phi^\pm_i(z)}$ satisfy the operator products
$$J^\pm_j(z)  e^{\eta \phi^\pm_i(w)} \sim \pm k\eta\delta_{i,j}  e^{\eta \phi^\pm_i(w)}(z-w)^{-1} + \frac{1}{\eta} \partial  e^{\eta \phi^\pm_i(w)},$$
$$  e^{\eta \phi^\pm_i(z)}  e^{\nu \phi^\mp_j(w)} \sim \delta_{i,j}(z-w)^{\pm k\eta\nu} :  e^{\eta \phi^\pm_i(z)} e^{\nu \phi^\pm_j(w)}:.$$
We will now introduce screening operators as the zero modes of fields of this type. For this define the even and odd Cartan subalgebra valued fields 
\begin{equation}
\phi \ = \ \sum_{i=1}^n\phi^+_iE_{ii}+\phi^-_iE_{n+i,n+i}, \qquad\qquad \psi  \ = \ \sum_{i=1}^n\psi^+_iE_{ii}+\psi^-_iE_{n+i,n+i}\, .
\end{equation}
Finally, the screening operators associated to our purely odd simple root system are as follows.
\begin{equation}
\begin{split}
Q_{\alpha_i} \ &= \ \text{Res}_z \bigl(:\alpha_i(\psi(z))e^{\alpha_i(\phi(z))}:\bigr) ,\\ 
Q_{\beta_i} \ &= \ \frac{1}{k}\text{Res}_z \bigl(:\beta_i(\psi(z))e^{\beta_i(\phi(z))}:\bigr).
\end{split}
\end{equation} 
It is convenient to change basis as follows. 
\begin{equation}
\begin{split}
Y_i(z) \ &= \ \phi^+_i(z)-\phi^-_i(z), \\
X_i(z) \ &= \ \frac{1}{2k}\bigl(\phi^+_i+\phi^-_i+\sum_{j=1}^{i-1}Y_j(z)-\sum_{j=i+1}^{n}Y_j(z)\bigr), \\
b^i(z) \ &= \ \psi^+_i(z)-\psi^-_i(z) ,\\
c^i(z) \ &= \ \frac{1}{2k}\bigl(\psi^+_i+\psi^-_i+\sum_{j=1}^{i-1}b^j(z)-\sum_{j=i+1}^{n}b^j(z)\bigr) .
\end{split}
\end{equation}
In this new basis, the nonregular OPEs are
\begin{equation}
Y_i(z)X_j(w)\ \sim \ \delta_{i,j}\ln(z-w),\qquad \qquad b^i(z)c^j(w)\ \sim \ \frac{\delta_{i,j}}{(z-w)}\,.
\end{equation}
The screening operators read in this basis
\begin{equation}
\begin{split}
Q_{\alpha_i} \ &= \ \text{Res}_z \bigl(:b^i(z)e^{Y_i(z)}:\bigr), \\ 
Q_{\beta_i} \ &= \ \text{Res}_z \bigl(:(c^i(z)-c^{i+1}(z))e^{k(X_i(z)-X_{i+1}(z))}:\bigr).
\end{split}
\end{equation} 
Let $M$ be the vertex algebra generated by $\partial Y_i(z),\partial X_i(z),b^i(z),c^i(z)$ for $i=1,\dots, n$. We have
\begin{lemma}\label{lemma:Q1}
Let
$$ N_i(z)=\partial X_i(z)-:b^i(z)c^i(z):,\qquad \ E_i(z)=\partial Y_i(z),$$
$$ \Psi_i^+(z)=b^i(z),\qquad \ \Psi^-_i(z)=\partial c^i(z)-:c^i(z)\partial Y_i(z):.$$
Then the vertex algebra generated by $N_i,E_i,\Psi^\pm_i$ is a homomorphic image of $V_1(\SLSA{gl}{1}{1})$. Moreover, this algebra is contained in the kernel of
$Q_{\alpha_i}$.
\end{lemma}
\begin{proof}
The nonregular operator products of $N_i,E_i,\Psi^\pm_i$  are
\begin{equation}
\begin{split}
N_i(z)E_i(w)\ &\sim \ \frac{1}{(z-w)^2}, \\
N_i(z)N_i(w)\ &\sim \ \frac{1}{(z-w)^2} ,\\
N_i(z)\Psi^\pm_i(w)\ &\sim \ \frac{\mp \Psi^\pm(w)}{(z-w)}, \\
\Psi^+_i(z)\Psi^-_i(w)\ &\sim \ -\frac{1}{(z-w)^2}-\frac{E_i(w)}{(z-w)}. \\
\end{split}
\end{equation}
which coincides with the operator product algebra of $V_1(\SLSA{gl}{1}{1})$.
$E_i,N_i,\Psi^+_i$ are obviously in the kernel of $Q_{\alpha_i}$. The statement for $\Psi^-_i$ follows from 
$$ :b^i(z)e^{Y_i(z)}:\Psi^-_i(w)\ \sim \ \frac{e^{Y_i(w)}}{(z-w)^2}\, .$$ 
\end{proof}
\begin{lemma}\label{lemma:Q2}
We have
\begin{equation*}
\begin{split}
&E_i+E_{i+1},N_i+N_{i+1}-\frac{1}{k}E_i,\Psi^\pm_i+\Psi^\pm_{i+1} \ \in \text{Ker}_{M}(Q_{\beta_i}),\\
&:\Psi_i^+N_i:+:\Psi_{i+1}^+N_{i+1}:-\frac{1}{k}\partial \Psi^+_i\ \in \text{Ker}_{M}(Q_{\beta_i}),\\
&:N_i\Psi^-_i:+:N_{i+1}\Psi^-_{i+1}:+\frac{1}{k}\bigl(:E_{i+1}\Psi^-_i:-:E_i\Psi^-_{i+1}:\bigr)-\frac{1}{k}\partial \Psi^-_i \ \in \text{Ker}_{M}(Q_{\beta_i}).
\end{split}
\end{equation*}
\end{lemma}
\begin{proof}
The first two lines are obvious, while the last one is a lengthy OPE computation. One needs
\begin{equation*}
\begin{split}
 :N_i(z)\Psi_i^-(z):\ &= \ :b^i(z)\partial c^i(z)c^i(z):-:c^i(z)\partial X_i(z)\partial Y_i(z):+\\
 &\qquad +:\partial c^i(z)\partial X_i(z):- :\partial c^i(z)\partial Y_i(z):+\frac{1}{2}\partial^2c^i(z).
 \end{split}
\end{equation*}
\end{proof}
We define the following fields in $M$
\begin{equation}\label{eq:fields}
\begin{split}
E(z)  &= -\sum_{i=1}^n E_i(z) ,\qquad\qquad
N(z)  =  -\sum_{i=1}^n N_i(z)+\frac{1}{k}\sum_{i=1}^n(n-i)E_i(z), \\
\Psi^+(z) &=  \sum_{i=1}^n \Psi^+_i(z) ,\qquad\qquad
\Psi^-(z) = \sum_{i=1}^n \Psi^-_i(z), \\
F^+(z)  &=  \sum_{i=1}^n:\Psi_i^+(z)N_i(z):-\frac{1}{k}\sum_{i=1}^n(n-i)\partial\Psi^+_i(z), \\
F^-(z)  &= \sum_{i=1}^n:N_i(z)\Psi_i^-(z):-\frac{1}{k}\sum_{i=1}^n(n-i)\partial\Psi^-_i(z)+\\
&\qquad+\frac{1}{2k}\Bigl(\sum_{1\leq i<j\leq n}:E_j(z)\Psi^-_i(z):-\sum_{1\leq j<i\leq n}:E_j(z)\Psi^-_i(z):\Bigr).
\end{split}
\end{equation}
\begin{thm}
$$E,N,\Psi^\pm,F^\pm \ \in \ \bigcap_{i=1}^n \text{Ker}_M(Q_{\alpha_i})\cap\bigcap_{i=1}^{n-1} \text{Ker}_M(Q_{\beta_i}).$$
\end{thm}
\begin{proof}
By Lemma \ref{lemma:Q1} we have $E,N,\Psi^\pm,F^\pm \ \in \ \bigcap_{i=1}^n \text{Ker}_M(Q_{\alpha_i})$ and Lemma \ref{lemma:Q2} implies
$E,N,\Psi^\pm,F^\pm \ \in \ \bigcap_{i=1}^{n-1} \text{Ker}_M(Q_{\beta_i})$.
\end{proof}
\begin{defn}
Let $\cW_{n,k}$ denote the vertex algebra generated by $E,N,\Psi^\pm,F^\pm$. We define $\cW_{n,\infty}$ to be $\lim_{k\rightarrow\infty} \cW_{n,k}$. 
\end{defn}
\begin{thm}
$\cV_{n}(\widehat{\cS\cD})\cong \cW_{n,\infty}$.
\end{thm}
\begin{proof}
We will construct the isomorphism explicitly. Let $\tilde\beta_i,\tilde\gamma_i,\tilde b_i,\tilde c_i$ be the generators of a rank $n$ $bc\beta\gamma$-system, and let $\tilde\phi^\pm_i,\tilde\phi_i$ be bosonic fields with OPE
$$ \tilde\phi^\pm_i(z)\tilde\phi^\pm_j(w)  \ \sim \ \pm\delta_{i,j}\ln(z-w),\qquad \qquad \tilde\phi_i(z)\tilde\phi_j(w)  \ \sim \ \delta_{i,j}\ln(z-w)\,.$$
Using the well-known bosonization isomorphism~\cite{FMS}, one obtains
$$ \tilde b_i(z)=e^{-\tilde\phi_i(z)},\qquad \tilde c_i(z)=e^{\tilde\phi_i(z)},\qquad:\tilde b_i(z)\tilde c_i(z):=-\partial\tilde\phi_i(z),$$ 
$$ \tilde \beta_i(z)=:e^{-\tilde\phi^-_i(z)+\tilde\phi^+_i(z)}\partial\tilde\phi^+_i(z):,\qquad \tilde \gamma_i(z)=e^{\tilde\phi^-_i(z)-\tilde\phi^+_i(z)},\qquad:\tilde\beta_i(z)\tilde \gamma_i(z):=\partial\tilde\phi^-_i(z).$$
We define 
$$Y_i=\tilde\phi_i-\tilde\phi^-_i,\qquad X_i=\tilde\phi^-_i-\tilde\phi^+_i\qquad \phi=\tilde\phi_--\tilde\phi^-_i+\tilde\phi^+_i.$$
Then the nonzero OPEs of these fields are 
$$ Y_i(z)X_j(w)  \ \sim \ \delta_{i,j}\ln(z-w),\qquad\qquad \phi_i(z)\phi_j(w)  \ \sim \ \delta_{i,j}\ln(z-w)\,.$$
Finally, we use bosonization again to obtain 
$$ b_i(z)=e^{-\phi_i(z)},\qquad c_i(z)=e^{\phi_i(z)},\qquad:b_i(z)c_i(z):=-\partial\phi_i(z).$$
Under this isomorphism, we get the following identifications
\begin{equation}
\begin{split}
E_i(z)\ &= \ \partial Y_i(z) \ = \ \partial \tilde\phi_i(z) -\partial\tilde\phi^-_i(z)\ = \ -:\tilde b_i(z)\tilde c_i(z):-:\tilde\beta_i(z)\tilde\gamma_i(z):, \\
N_i(z)\ &= \ \partial Y_i(z)-:b_i(z)c_i(z): \ = \ \partial \tilde\phi_i(z)\ = \ -:\tilde b_i(z)\tilde c_i(z):, \\
\Psi^+_i(z) \ &= \ b_i(z) \ = \ e^{-\phi_i(z)}\ = \ e^{-\tilde\phi_i(z)+\tilde\phi^-_i(z)-\tilde\phi^+_i(z)}\ = \ :\tilde b_i(z)\tilde\gamma_i(z):,\\
\Psi^-_i(z) \ &= \ \partial c_i(z)-:c_i(z)\partial Y_i(z): \ = \ :e^{\phi_i(z)}(\partial\phi_i(z)-\partial Y_i(z)):\\ &= \ 
:e^{\tilde\phi_i(z)-\tilde\phi^-_i(z)+\tilde\phi^+_i(z)}\partial\tilde\phi^+_i(z):\ = \ :\tilde c_i(z)\tilde\beta_i(z):.\\
\end{split}
\end{equation}
and hence 
\begin{equation*}
\begin{split}
E(z) &= \sum_{i=1}^n :\tilde b_i(z)\tilde c_i(z):+:\tilde\beta_i(z)\tilde\gamma_i(z): ,\qquad\qquad
N(z)= \sum_{i=1}^n :\tilde b_i(z)\tilde c_i(z):, \\
\Psi^+(z) &= \sum_{i=1}^n :\tilde b_i(z)\tilde \gamma_i(z):,\qquad\qquad
\Psi^-(z)=  \sum_{i=1}^n :\tilde c_i(z)\tilde \beta_i(z):, \\
F^+(z)  &=  \sum_{i=1}^n :\Psi^+_i(z)N_i(z): \ = \ \sum_{i=1}^n: (:\tilde b_i(z)\tilde\gamma_i(z):)(:\tilde b_i(z)\tilde c_i(z):): 
= \ -\sum_{i=1}^n :\tilde b_i(z)\partial\tilde\gamma_i(z):,\\ 
F^-(z)  &=  \sum_{i=1}^n :N_i(z)\Psi^-_i(z): \ = \ \sum_{i=1}^n:(:\tilde b_i(z)\tilde c_i(z):)(:\tilde c_i(z)\tilde \beta_i(z):): 
= \ -\sum_{i=1}^n :\tilde \beta_i(z)\partial\tilde c_i(z): .\\ 
\end{split}
\end{equation*}
But these fields are by Lemma \ref{weakfg} a generating set of $\cV_{n}(\widehat{\cS\cD})$. 
\end{proof}
The operator product algebra of $\cW_{2,k}$ can be computed explicitly. 
For this, we choose a slightly different basis from \eqref{eq:fields}. Let $n=2$, and define
$$
 G^\pm\ = \ F^\pm \pm\Bigl(\frac{1}{2k}-\frac{1}{2}\Bigr)\partial \Psi^\pm.$$
There will be two additional fields, a Virasoro field of central charge zero,
\begin{equation*}
\begin{split}
T\ = \ &:E_1(z)N_1(z):+:E_2(z)N_2(z):-:\Psi_1^+(z)\Psi_1^-(z):-:\Psi_2^+(z)\Psi_2^-(z):\\
&\ -\frac{1+k}{2k}\partial E_1(z)+\frac{1-k}{2k}\partial E_2(z),
\end{split}
\end{equation*}
and another dimension two field
\begin{equation*}
\begin{split}
 H\ = \ &-\frac{1}{2}\bigl( :N_1(z)N_1(z):+:N_2(z)N_2(z):-:E_1(z)N_1(z):-:E_2(z)N_2(z):\\
 &\ +:\Psi_1^+(z)\Psi_1^-(z):-:\Psi_2^+(z)\Psi_2^-(z):\bigr)+
\frac{1}{2k}\bigl(:E_1(z)N_2(z):-:E_2(z)N_1(z):\\
&\ +:\Psi_1^+(z)\Psi_2^-(z):-:\Psi_2^+(z)\Psi_1^-(z):+\frac{1}{k}:E_1(z)E_2(z):+\partial N_1(z)-\partial N_2(z)\bigr)\\
&\ -\frac{1}{8k^2}\bigl((2k^2+2k+1)\partial E_1(z)+(2k^2-2k+1)\partial E_2(z)\bigr).
\end{split}
\end{equation*}
Then $E,N,\Psi^\pm$ have the operator product algebra of $\AKMSA{gl}{1}{1}$ at level two, and $G^\pm$ are Virasoro primaries of dimension two, while $H$ is the {\emph{partner}} of $T$,
$$
T(z)H(w)\ \sim \ \Bigl(\frac{3}{k^2}-1\Bigr)\frac{1}{(z-w)^4}+\frac{3}{4k^2}\frac{E(w)}{(z-w)^3}+\frac{2H(w)}{(z-w)^2}+\frac{\partial H(w)}{(z-w)}.
$$ 
In addition the operator products of the dimension two fields with the currents are
\begin{equation*}
\begin{split}
N(z)H(w)\ &\sim \ \frac{3}{2k^2}\frac{1}{(z-w)^3}-\Bigl(\frac{1}{4}-\frac{3}{4k^2}\Bigr)\frac{E(w)}{(z-w)^2},\qquad 
N(z)G^\pm(w)\ \sim \ \frac{\pm G^\pm(w)}{(z-w)},\\
\Psi^\pm(z)H(w)\ &\sim \ -\frac{G^\pm(w)}{(z-w)},\qquad
\Psi^\pm(z)G^\mp(w)\ \sim \ \frac{N(w)}{(z-w)^2}\pm\frac{T(w)}{(z-w)},\\
E(z)H(w)\ &\sim \ -\frac{N(w)}{(z-w)^2},\qquad
E(z)G^\pm(w)\ \sim \ -\frac{\pm\psi^\pm(w)}{(z-w)^2}.\\
\end{split}
\end{equation*}  
Introduce the following normally ordered polynomials in the currents and their derivatives
\begin{equation*}
\begin{split}
X_0=&\frac{1}{2}\Bigl(2:HE:-2:TE:-2:TN:-2:G^+\Psi^-:-2:G^-\Psi^+:+:\partial\Psi^-\Psi^+:+:\partial\Psi^+\Psi^-:\\
&+:\partial EN:-2:N\Psi^+\Psi^-:+:NNE:-:E\Psi^+\Psi^-:+:NEE:\Bigr)\\
&-\frac{1}{8k^2}\Bigl((1-2k^2)\partial^2E+(3-2k^2):\partial EE:+(1-k^2):EEE:\Bigr),
\end{split}
\end{equation*}
\begin{equation*}
\begin{split}
X^+= &\frac{1}{2}\Bigl(:N\partial\Psi^+:-2:H\Psi^+:-2NG^+:+:T\Psi^+:-:EG^+:-:NN\Psi^+:-:NE\Psi^+:\Bigr)\\
&-\frac{1}{8k^2}\Bigl((2+2k^2)\partial^2\Psi^+-:\partial E\Psi^+:-(2+2k^2):E\partial\Psi^+:-(1-k^2):EE\psi^+:\Bigr),
\end{split}
\end{equation*}
\begin{equation*}
\begin{split}
X^-= &\frac{1}{2}\Bigl(2NG^-:-:N\partial\Psi^-:-2:H\Psi^-:+:T\Psi^-:+:EG^-:-:NN\Psi^-:-:NE\Psi^-:\Bigr)\\
&+\frac{1}{8k^2}\Bigl((2+2k^2)\partial^2\Psi^++5:\partial E\Psi^-:-(2+2k^2):E\partial\Psi^-:+(1-k^2):EE\psi^-:\Bigr),
\end{split}
\end{equation*}
\begin{equation*}
\begin{split}
X_2=&3\partial^2N+(2+2k^2)\partial^2E+4:\partial\psi^-\psi^+:-4:\partial\psi^+\psi^-:+
 4:\partial NE:+4:\partial EN:+2:\partial EE:.
\end{split}
\end{equation*}

Then the operator products of the dimension two fields with themselves are
\begin{equation*}
\begin{split}
H(z)H(w)\ &\sim \ -\frac{1}{4k^2}\frac{1}{(z-w)^2}\Bigl(2\partial E(w)+3\partial N(w)-(2k^2+2)T(w)+4:N(w)E(w):+\\
&\qquad :E(w)E(w):-4:\psi^+(w)\psi^-(w):\Bigr)-\frac{1}{8k^2}\frac{X_2(w)}{(z-w)}, \\
H(z)G^+(w)\ &\sim \ \Bigl(\frac{1}{2}-\frac{3}{4k^2}\Bigr)\frac{\Psi^+(w)}{(z-w)^3}+\Bigl(\frac{1}{4}-\frac{3}{4k^2}\Bigr)\frac{\partial\Psi^+(w)}{(z-w)^2}+\frac{G^+(w)+X^+(w)}{(z-w)},\\
H(z)G^-(w)\ &\sim \ \Bigl(\frac{1}{2}-\frac{9}{4k^2}\Bigr)\frac{\Psi^-(w)}{(z-w)^3}+\Bigl(\frac{1}{4}-\frac{3}{4k^2}\Bigr)\frac{\partial\Psi^-(w)}{(z-w)^2}+\frac{G^-(w)+X^-(w)}{(z-w)},\\
G^+(z)G^-(w)\ &\sim \ -\Bigl(1-\frac{3}{k^2}\Bigr)\frac{1}{(z-w)^4}-\Bigl(\frac{1}{2}-\frac{3}{2k^2}\Bigr)\frac{E(w)}{(z-w)^3}-\frac{1}{4}\frac{\partial E(w)-8H(w)}{(z-w)^2}\\
&\qquad+(H(w)+X_0(w))(z-w)^{-1}.
\end{split}
\end{equation*}
We see 
\begin{prop}
$\cW_{2,k}$ is strongly generated by $E,N,\Psi^\pm,T,H,G^\pm$.
\end{prop}

\section{The relationship between $\cB_{n,k}$ and $\cW_{n,k}$} 
Recall the algebra $\cB_{n,k}$ that we constructed in Section \ref{sectiondefsv}, which has a minimal strong generating set consisting of $4n$ fields, for generic values of $k$. In this section we show that for $n=2$, $\cW_{2,k+2}$ and $\cB_{2,k}$ have the same generators and OPE relations. More generally, we conjecture that $\cW_{n,k+n}$ is isomorphic to $\cB_{n,k}$ for all $n$ and $k$.

Recall that $V_k(\gg\gl(n))$ has a strong generating set $\{X^{ij}|1 \leq i,j\leq n\}$ satisfying
\begin{equation}
X^{ij}(z)X^{lm}(w)\ \sim \ \frac{k\delta_{j,l}\delta_{i,m}}{(z-w)^2}+\frac{\delta_{j,l}X^{im}(w)-\delta_{i,m}X^{lj}(w)}{(z-w)} .
\end{equation}
Recall the $bc\beta\gamma$-system $\cF = \cE\otimes \cS$ of rank $n$. 
There is a map $V_1(\gg\gl(n)) \ra \cE$ sending $X^{ij} \mapsto :c^ib^j:$ and a map $V_{-1}(\gg\gl(n)) \ra \cS$ sending $X^{ij} \mapsto - :\gamma^i\beta^j:$. These combine to give us a map $V_0(\gg\gl(n)) \ra \cF$ sending $X^{ij} \mapsto c^ib^j: - :\gamma^i\beta^j$.

A straightforward computation shows that $\cB_{n,k} = \text{Com}(V_k(\gg\gl(n)),V_k(\gg\gl(n))\otimes \cF)$ contains the following elements: 

 $$\tilde \Psi^-=\sum_{l=1}^n:c^l\beta^l:,\qquad \ \tilde\Psi^+=-\sum_{l=1}^n:\gamma^l b^l:,$$ 
 $$\tilde E=-\sum_{l=1}^n:c^l b^l:+:\gamma^l\beta^l:,\qquad \
 \tilde N=\sum_{l=1}^n\frac{2}{k}X^{ll}-:c^l b^l:+:\gamma^l\beta^l:,$$
$$ \tilde F^-=\frac{1}{k}\sum_{1\leq j,l\leq n} :X^{jl}c^l\beta^j:+\sum_{l=1}^n :c^l\partial \beta^l:,\qquad \
 \tilde F^+ = \frac{1}{k}\sum_{1\leq j,l\leq n} :X^{jl}\gamma^l b^j:+\sum_{l=1}^n :\gamma^l\partial b^l:.$$
 By Lemma \ref{weakfg}, the elements of $\cB_{n,\infty} = \cV_{n}(\widehat{\cS\cD})$ corresponding to these six elements under $\phi_k: \cB_{n,k} \ra \cV_{n}(\widehat{\cS\cD})$, are a generating set for $\cV_{n}(\widehat{\cS\cD})$. By Corollary \ref{genbk}, these six fields generate $\cB_{n,k}$ for generic values of $k$.

\begin{thm} For generic values of $k$, $\cW_{2,k+2}$ and $\cB_{2,k}$ have the same OPE algebra. \end{thm}

\begin{proof} This is a computer computation, where the field identification is given by, with $s=(1+k)/(4+2k)$,
\begin{equation}
\begin{split}
E\ &\rightarrow\  \tilde E, \qquad N\ \rightarrow\ \tilde N-\frac{\tilde E}{k},\qquad \Psi^\pm\ \rightarrow\  \tilde \Psi^\pm, \\
G^+\ &\rightarrow\ (4s-1)\tilde F^++s\partial\tilde\Psi^++(2s-1):\tilde N\tilde\Psi^+:,\\
 G^-\ &\rightarrow\ (4s-1)\tilde F^--(3s-1)\partial\tilde\Psi^--(2s-1):\tilde N\tilde\Psi^-:. 
\end{split}
\end{equation}
\end{proof}

\begin{remark}
There exist other realizations of $\cW_{2,k}$.
Let $E_{ij}$ for $1\leq i,j\leq 4$ be the basis of $\SLSA{gl}{2}{2}$ of Definition \ref{def:glnn}, and we denote the corresponding fields of $V_k(\SLSA{gl}{2}{2})$ by $E_{ij}(z)$.
Then, we find that for level $k=-2$ the fields 
\begin{equation*}
\begin{split}
E'&=-\frac{1}{2}(\sum_{i=1}^4 E_{ii}),\qquad\qquad\ \ \, N'=\frac{1}{2}(E_{11}+E_{22}-E_{33}-E_{44}),\\ 
{\Psi^+}'&=\frac{1}{\sqrt{-2}}(E_{13}+E_{24}) ,\qquad {\Psi^-}'=\frac{1}{\sqrt{-2}}(E_{31}+E_{42}),\\ 
\end{split}
\end{equation*}
\begin{equation*}
\begin{split}{G^+}'&= \frac{1}{\sqrt{-2}}(:E_{12}E_{23}:+:E_{21}E_{14}:+:(E_{11}-E_{22})(E_{13}-E_{24}):)\\
&\qquad-\frac{1}{2}\partial{\Psi^+}'-\frac{1}{2}:N'{\Psi^+}':+\frac{1}{4}:E'{\Psi^+}':,\\
{G^-}'&= \frac{1}{\sqrt{-2}}(:E_{12}E_{41}:+:E_{21}E_{32}:+:(E_{11}-E_{22})(E_{31}-E_{42}):)\\
&\qquad-\frac{1}{2}\partial{\Psi^-}'+\frac{1}{2}:N'{\Psi^-}':-\frac{1}{4}:E'{\Psi^-}':.\\
\end{split}
\end{equation*}
are elements of $\text{Com}(V_{0}(\SLA{sl}{2}),V_{-2}(\SLSA{gl}{2}{2}))$ and satisfy the operator product algebra of $\cW_{2,-1}$ where the field identification is given by $X \rightarrow X'$, for $X = E, N, \Psi^{\pm}, G^{\pm}$.

\end{remark}

\end{document}